\documentclass{amsart}
\usepackage[top=1.25in, bottom=1.25in, left=1.25in, right=1.25in]{geometry}

\usepackage{amsmath,amssymb,amsthm,enumitem,hyperref}
\usepackage{thm-restate}
\usepackage[pdftex]{graphicx}

\usepackage[all,cmtip]{xy}
\usepackage{tikz-cd}
\usepackage{tensor}

\theoremstyle{plain}
\newtheorem{theorem}{Theorem}[section]
\newtheorem{lemma}[theorem]{Lemma}
\newtheorem{cor}[theorem]{Corollary}
\newtheorem{prop}[theorem]{Proposition}
\newtheorem{question}[theorem]{Question}
\newtheorem{addendum}[theorem]{Addendum}
\newtheorem{conjecture}[theorem]{Conjecture}
\newtheorem*{lemma*}{Lemma}
\newtheorem*{cor*}{Corollary}
\newtheorem*{theorem*}{Theorem}

\theoremstyle{definition}
\newtheorem{definition}[theorem]{Definition}

\newtheorem{assumption}[theorem]{Assumption}
\newtheorem{notation}[theorem]{Notation}
\newtheorem{remark}[theorem]{Remark}

\theoremstyle{remark}

\let\oldproofname=\proofname
\renewcommand{\proofname}{\rm\bf{\oldproofname}}

\newcommand{\R}{\mathbb R}

\newcommand{\Z}{\mathbb Z}
\newcommand{\C}{\mathbb C}
\newcommand{\N}{\mathbb N}

\newcommand{\phii}{\varphi}

\newcommand{\op}[1]{\operatorname{{#1}}}

\newcommand{\mc}[1]{\mathcal{{#1}}}

\newcommand{\Hom}{\op{Hom}}

\newcommand{\gprod}[3]{\left( #1 | #2 \right)_{#3} }
\newcommand{\acts}{\curvearrowright}

\newcommand{\PD}[1]{PD$(#1)$}

\newcommand{\co}{\colon}
\newcommand{\leftQ}[2]{\left.\raisebox{-.2em}{$#2$}\middle\backslash\raisebox{.2em}{$#1$}\right.}
\newcommand{\supp}{\op{supp}}
\begin{document}
    \title{Cohomology and the Bowditch boundary}

    \author{Jason F. Manning}
\address{Department of Mathematics, 310 Malott Hall, Cornell University, Ithaca, NY 14853}
    \email{jfmanning@math.cornell.edu}
    \author{Oliver H. Wang}
    \address{Department of Mathematics, University of Chicago, Chicago, IL 60637}
    \email{oliverwang@uchicago.edu}
    \begin{abstract}
      We give a group cohomological description of the \v{C}ech cohomology of the Bowditch boundary of a relatively hyperbolic group pair, generalizing a result of Bestvina--Mess about hyperbolic groups.  In case of a relatively hyperbolic Poincar\'e duality group pair, we show the Bowditch boundary is a homology manifold.  For a three-dimensional Poincar\'e duality pair, we recover the theorem of Tshishiku--Walsh stating that the boundary is homeomorphic to a two-sphere.
    \end{abstract}
\thanks{This project was supported by the National Science Foundation, grant DMS-1462263, and by a grant from the Simons Foundation (524176, JFM)}
    \maketitle

\section{Introduction}
  In Gromov's influential essay on hyperbolic groups \cite{Gromov87}, he introduced a compactification, now called the \emph{Gromov boundary} and initiated a study of the dynamics of a hyperbolic group on its boundary.  Later, Bowditch showed that boundaries of hyperbolic groups are completely characterized by this dynamical structure, which mirrors the dynamics of a convex cocompact Kleinian group \cite{Bowditch98}.  Bestvina and Mess in \cite{BM91} showed that this boundary also contains \emph{algebraic} information about the group, by showing, for any hyperbolic group $G$, any $k\geq 1$ and any ring $A$, there is an isomorphism
\[H^k(G;AG)\cong\check{H}^{k-1}(\partial G;A)\]
between the group cohomology of $G$ and the \v{C}ech cohomology of its boundary.  (Throughout this paper \v{C}ech cohomology is reduced, and $A$ is some fixed  ring.)
Note there is a natural map from $\partial G$ to the space of ends of a hyperbolic group, collapsing components of $\partial G$ to points, so this is consistent with the isomorphism $H^1(G;\Z G) \cong \check{H}^0(\mathrm{Ends}(G);\Z)$, as described, for example in \cite[Chapter 13]{Geoghegan}.  

  Gromov also introduced the idea of a \emph{relatively hyperbolic} group pair in \cite{Gromov87}, which attracted little attention until the work of Farb in the mid-90s \cite{Farb98}.\footnote{In that work a relatively hyperbolic pair is called a \emph{strongly} relatively hyperbolic pair.}   In Bowditch's 1999 preprint, published as \cite{Bowditch12}, a boundary was described for such a pair.  The dynamics of a relatively hyperbolic group pair acting on this boundary mirror the dynamics of a geometrically finite Kleinian group acting on its limit set.  Yaman proved an analogue of Bowditch's result about hyperbolic groups, that a group action on a metrizable compactum satisfying certain dynamical criteria must be relatively hyperbolic, and the compactum must be equivariantly homeomorphic to the Bowditch boundary \cite{Yaman04}.  In this paper we show an analogue of Bestvina and Mess's result in the relatively hyperbolic setting.  
  \begin{restatable}{theorem}{maintheorem}\label{thm:maintheorem}
    If $(G,\mathcal{P})$ is relatively hyperbolic and type $F_\infty$, then for every $k$, there is an isomorphism of $AG$-modules
\begin{equation}\label{mainisomorphism} H^k(G,\mathcal{P};AG)\to \check{H}^{k-1}(\partial(G,\mathcal{P});A). 
\end{equation}
  \end{restatable}
 We recall Bieri and Eckmann's definition of relative group cohomology \cite{BE78} in Section \ref{ss:relative cohomology}.
  A group is \emph{type $F$ (resp. type $F_\infty$)} if it admits an Eilenberg-MacLane classifying space with finitely many cells (resp. finitely many cells in each dimension).  A pair $(G,\mc{P})$ is type $F$ or $F_\infty$ if both $G$ and every $P\in \mc{P}$ are.

  Theorem \ref{thm:maintheorem} implies that in some sense all the high-dimensional group cohomology of $G$ comes from its peripheral groups $\mc{P}$ (see Corollary \ref{cor:highdcohomisperipheral}).

  Theorem \ref{thm:maintheorem} was previously obtained by Kapovich in case $G$ is geometrically finite Kleinian and $\mc{P}$ is the collection of maximal parabolic subgroups, up to conjugacy in $G$ \cite{K09}.  In that case the Bowditch boundary $\partial(G,\mc{P})$ is equivariantly homeomorphic to the limit set $\Lambda(G)$.  

  \begin{question}
    Do the isomorphisms \eqref{mainisomorphism} hold without the assumption that the pair is $F_\infty$?
  \end{question}

  Poincar\'e duality group pairs are of particular interest.  In this context, we show the following analogue of \cite[Theorem 2.8]{B96}.
\begin{restatable}{theorem}{pdpair}\label{thm:pdpair}
  Suppose $(G,\mc{P})$ is relatively hyperbolic and type $F$.  The following are equivalent:
  \begin{enumerate}
  \item $(G,\mc{P})$ is a \PD{n} pair.
  \item $\partial(G,\mathcal{P})$ is a homology $(n-1)$--manifold and an integral \v{C}ech cohomology $(n-1)$--sphere.
  \end{enumerate}
\end{restatable}
In the particular case that $n=3$, then we recover a result of Tshishiku and Walsh \cite{TshWa}.  Namely,
$(G,\mc{P})$ is a \PD{3} pair if and only if $\partial(G,\mc{P})$ is homeomorphic to a $2$--sphere  (see Corollary \ref{cor:2sphere}).  
\subsection{Outline}
In Section \ref{sec:prelim}, we recall some basic facts about relative hyperbolicity and various cohomology theories which occur in the paper.  Some of this is expanded on further in an appendix.  

In Section \ref{sec:N-connected}, we construct, for an $F_\infty$ relatively hyperbolic group pair, a sequence of more and more highly connected metric spaces on which the pair acts in a cusp uniform manner.  We show that the Bowditch boundary compactifies these spaces as kind of weak $Z$--set (see Corollary \ref{cor:weak Zset}).  In case the pair $(G,\mc{P})$ is type $F$, the Bowditch boundary is an honest $Z$--set (Theorem \ref{thm:zset}).  In any case, it has enough of the properties of a $Z$--set that we can establish Theorem \ref{thm:maintheorem}, which we do in subsection \ref{ss:mainthm}.  

In Section \ref{sec:pdpair}, we restrict attention to relatively hyperbolic \PD{n} pairs and adapt an argument of Bestvina from \cite{B96} to show Theorem \ref{thm:pdpair}.  The key idea here, also used in the final section, is that the cellular/simplicial chain complex of the cusped space we build is ``regular'', in the sense that coboundaries with support near a point at infinity are coboundaries of cochains which are also supported near that point at infinity.

Finally in Section \ref{sec:dimension}, we show (Theorem \ref{topdimthm}) that the topological dimension of the Bowditch boundary can be computed from relative group cohomology, at least with the hypothesis that $\op{cd}(G)<\op{cd}(G,\mc{P})$.  We conjecture the hypothesis is not necessary (Conjecture \ref{conj:dimension}).
\subsection{Acknowledgments}  
The authors thank Ken Brown, Ross Geoghegan, Mike Mihalik, Alessandro Sisto and Jim West for useful conversations.  Thanks also to the referee for useful comments.

\section{Preliminaries}\label{sec:prelim}
This section of background material can be skimmed on first reading and referred back to as necessary.
\subsection{Gromov hyperbolic spaces}
We refer to \cite[III.H]{BH99} for more detail about hyperbolic metric spaces.  We review just enough to fix notation.  Metrics will mostly be written $d(\cdot,\cdot)$, with the specific metric space evident from context.  Sometimes the metric on a space $S$ will be written $d_S(\cdot,\cdot)$.

Let $\delta\geq 0$.  A \emph{$\delta$--hyperbolic space} is a geodesic metric space so that every geodesic triangle is \emph{$\delta$--thin}, i.e. the canonical map to the comparison tripod has fibers of diameter at most $\delta$.  The \emph{Gromov product} $\gprod{x}{y}{z} = \frac{1}{2} (d(z,x)+d(z,y)-d(x,y))$ is the distance from the center of this tripod to the comparison point for $z$.  In a $\delta$--hyperbolic space, $\gprod{x}{y}{z}$ is within $\delta$ of the distance between $z$ and any geodesic joining $x$ to $y$.

A geodesic space is \emph{Gromov hyperbolic} if it is $\delta$--hyperbolic for some $\delta$.
If $X$ is Gromov hyperbolic, and $z\in X$ a basepoint, we say that $\{x_i\}_{i\in \N}$ \emph{tends to infinity} if $\lim_{i,j\to\infty}\gprod{x_i}{x_j}{z}=\infty$.  The \emph{Gromov boundary} $\partial X$ of $X$ is the set of equivalence classes of sequences which go to infinity, where $\{x_i\}_{i\in \N}\sim \{y_i\}_{i\in\N}$ if $\lim_{i,j\to\infty}\gprod{x_i}{y_j}{z} = \infty$.  Whether a sequence goes to infinity is independent of the basepoint $z$.

The Gromov product extends to the boundary; if $\mathbf{x}$ and $\mathbf{y}$ are points in $\partial X$,
\[ \gprod{\mathbf{x}}{\mathbf{y}}{z} = \sup\liminf_{i,j\to\infty}\gprod{x_i}{y_j}{z} \]
where the supremum is taken over all $\{x_i\}_{i\in \N}$ and $\{y_i\}_{i\in \N}$ so that  $\mathbf{x} =\left[ \{x_i\}_{i\in \N}\right]$ and $\mathbf{y} =\left[ \{y_i\}_{i\in \N}\right]$.  Points inside $X$ can be represented by constant sequences.  Then a sequence $\{x_i\}_{i\in \N}$ \emph{converges to} $\mathbf{y}\in \partial X$ if 
\[ \lim_{i\to\infty} \gprod{x_i}{\mathbf{y}}{z} = \infty .\]
This determines a topology on $\overline{X}=X\cup\partial X$ which is independent of the choice of basepoint $z$.  Quasi-isometries (coarsely bi-Lipschitz functions with coarsely dense image) of Gromov hyperbolic spaces extend to homeomorphisms of their boundaries.

\begin{definition}\label{def:Rips complex}
  If $Y$ is any metric space, and $D>0$, the \emph{Rips complex on $Y$ with parameter $D$}  is the simplicial complex $R_D(Y)$ whose vertex set is equal to $Y$, so that distinct points $\{x_0,\ldots,x_n\}$ span an $n$--simplex if $\max\{d_X(x_i,x_j)\}\leq D$.
\end{definition}

\begin{definition}\label{def:compactification of rips}
Suppose $Y$ is a hyperbolic metric space and $Y_0\subset Y$ be a locally finite $C$--dense subset. (Locally finite means every ball meets finitely many elements of $Y_0$; $C$--dense means that every point in $Y$ is within $C$ of some point of $Y_0$).  To topologize $\overline{R_D(Y_0)}=R_D(Y_0)\cup\partial Y$, it suffices to describe a neighborhood basis for a point $z\in\partial Y$.  Fix a basepoint $y_0\in Y_0$.  For each $n\geq 0$, let $V(z,n)$ be the subcomplex of $R_D(Y_0)$ spanned by $\{y\in Y_0\mid \gprod{y}{z}{y_0}\geq n\}$ (using the Gromov product in $Y$).  These sets give a basis of closed neighborhoods of $z$.
\end{definition}

We need the following refinement of \cite[III.$\Gamma$.3.23]{BH99}.
  \begin{lemma}\label{lem:contract nearby}
    Let $Y$ be a $\delta$--hyperbolic space, let $Y_0\subset Y$ be a $C$--dense subset, and let $R = R_D(Y_0)$ where $D\geq 4\delta + 6C$.  Let $L\subset R$ be any finite $k$--dimensional complex, and $x_0$ a vertex of $L$.  Then there is a contraction of $L$ to $x_0$ in the subcomplex of $R^{(k+1)}$ spanned by elements of $Y_0$ which lie within $C+\delta$ of some $Y$--geodesic joining $x_0$ to some vertex of $L$.
  \end{lemma}
  \begin{proof}
    A close reading of the proof of \cite[III.$\Gamma$.3.23]{BH99} yields the statement, as we now explain.  

    Given a $k$--complex $L$ in $R$ and a vertex $x_0$ of $L$, one homotopes the complex $L$ closer and closer to $x_0$, beginning with cells adjacent to a vertex $v$ which is farthest in $Y$ from $x_0$.  To accomplish this homotopy, one must find an element $v_1\in Y_0$ which is within $R$ of every vertex of $L$ adjacent to $v$, and so that $d_Y(v_1,x_0)$ is smaller than $d_Y(v,x_0)$ by a definite amount.  Then for each $k$--simplex $\sigma$ of $L$ containing $v$, there is a $(k+1)$--simplex containing $\sigma$ and the vertex $v_1$, and we can homotope $L$ across this $(k+1)$--simplex.  Ultimately the complex $L$ will be homotoped into a single simplex $\sigma_0$ of which $x_0$ is a vertex, and it then can be contracted easily in the $(k+1)$--skeleton of that simplex.

  The vertex $v_1$ is found by the following procedure:  Let $\gamma_v$ be a $Y$--geodesic from $x_0$ to $v$, let $y_v$ be a point on $\gamma_v$ at distance $D/2$ from $v$, and let $v_1$ be any point in $Y_0$ within $C$ of $y_v$.  (See \cite[III.$\Gamma$.3.23]{BH99} for the proof that this vertex satisfies the necessary adjacencies.)

  If $v_1$ is not now in a $D$--ball about $x_0$ it will eventually be the farthest vertex again, and we obtain another vertex $v_2$, and so on.  Given $v$, let $v_1,\ldots,v_n$ be the collection of vertices obtained in this way, and note that these are precisely the vertices $v$ passes through on the way to the simplex $\sigma_0$.

  We claim that every vertex $v_i$ lies within $C+\delta$ of $\gamma_v$, and we prove this by induction on $i$.  The vertex $v_1$ lies within $C$ of $\gamma_v$, so we can get started.  Suppose that $v_i$ lies within $C+\delta$ of $\gamma_v$, and let $\gamma_i$ be the geodesic from $v_i$ to $x_0$ we use to choose $v_{i+1}$, and let $y_i$ be the point on $\gamma_i$ at distance $D/2$ from $v_i$.
  \begin{figure}[htbp]
    \centering
    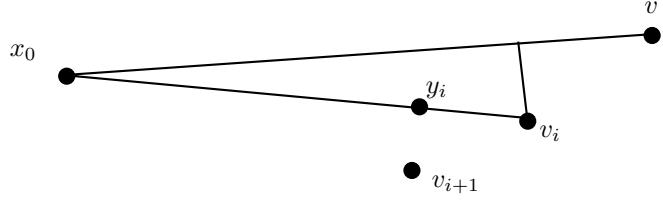
    \caption{The point $v_{i+1}$ is not too far from $\gamma_v$.}
    \label{fig:ripscontract}
  \end{figure}
    See Figure \ref{fig:ripscontract}.  Note that $v_i$ is at most distance $C+\delta$ from some point $z_i\in \gamma_v$ by induction, and that $d(v_i,y_i) = D/2 >C+\delta$.  By the thinness of the triangle with vertices $x_0, v_i$, and $z_i$, we have that $y_i$ is at most $\delta$ from $\gamma_v$.  Thus $v_{i+1}$ is distance at most $C+\delta$ from $\gamma_v$.
  \end{proof}
  
With this lemma, we prove the following statement.

\begin{prop}\label{prop:condition iii for Rips complex}
Suppose $Y$ is $\delta$-hyperbolic.  Let $Y_0\subseteq Y$ be a $C$-dense locally finite subset. Let $R$ be as in \ref{lem:contract nearby} and let $\bar{R}=R\cup\partial Y$. Then, for every $z\in\partial Y$ and every neighborhood $U\subseteq\bar{R}$ of $z$, there is a neighborhood $V\subseteq U$ of $z$ such that, for all $i$, every map $S^i\rightarrow V\setminus\partial Y$ is nullhomotopic in $U\setminus\partial Y$.
\end{prop}
	\begin{proof}
	The proof of this proposition is essentially contained in the proof of \cite[Theorem 1.2]{BM91}. There, the authors treat the case that $Y$ is a graph.
	
	Let $z\in\partial Y$ and $U\subseteq\bar{R}$ be a neighborhood of $z$. Fix a point $y_0\in Y$. There is a constant $c>0$ such that $U$ contains the closure of the subcomplex of $R$ spanned by the vertices $v$ which satisfy $(v|z)_{y_0}\ge c$. Let $V$ be the closure of the subcomplex spanned by the vertices $v$ such that $(v|z)_{y_0}\ge2c+C+4\delta$.
	
	Let $f:S^i\rightarrow V\setminus\partial Y$ be a map. we may assume that the image of $f$ is contained in a finite subcomplex of $R$. Suppose $v_1$ and $v_2$ are vertices of this subcomplex. By Lemma \ref{lem:contract nearby}, it suffices to show that, if $v$ is a vertex with $d(v,v_1)+d(v,v_2)\le d(v_1,v_2)+2(C+\delta)$ then $v\in U$. Let $v$ be such a vertex. We have $(v_1|v_2)_{y_0}\ge\min\{(v_1|z)_{y_0},(v_2|z)_{y_0}\}-\delta\ge2c+C+3\delta$. Our assumption on $v$ gives $(v_1|v)_{y_0}+(v_2|v)_{y_0}\ge(v_1|v_2)_{y_0}+d(v,y_0)-C-\delta\ge2c+2\delta$. We may assume (possibly switching $v_1$ and $v_2$) that $(v_1|v)_{y_0}\ge c+\delta$. Then $(v|z)_{y_0}\ge\min\{(v_1|z)_{y_0},(v_1|v)_{y_0}\}\ge c$ as desired.
	\end{proof}

\subsection{Cusped spaces, relative hyperbolicity, and the Bowditch boundary}
\subsubsection{The combinatorial cusped space}
Let $\Gamma$ be a graph.
\begin{definition}\label{def: comb horoball} The \emph{combinatorial horoball based on} $\Gamma$ is the graph with vertex set $\Z_{\geq 0}\times v(\Gamma)$ and the following edges
\begin{enumerate}
\item A \textit{vertical} edge between $(n,v)$ and $(n+1,v)$
\item A \textit{horizontal} edge between $(n,v)$ and $(n,w)$ whenever $d_\Gamma(v,w)\le 2^n$
\end{enumerate}
The combinatorial horoball is denoted $CH(\Gamma)$ and is endowed with a metric giving all edges length $1$.
Define the \textit{depth} of a vertex $(n,v)$ to be $n$ and extend this linearly over the edges.
\end{definition}

\begin{definition}
  In this paper, a \emph{group pair} $(G,\mc{P})$ is a finitely generated group $G$ together with a finite collection $\mc{P}$ of finitely generated proper subgroups of $G$.
\end{definition}

\begin{definition}\label{def:combinatorial cusped}
  Let $(G,\mc{P})$ be a group pair.  Suppose $S$ is a finite generating set for $G$ which contains finite generating sets for each $P\in \mc{P}$.  (Such a generating set is called a \emph{compatible} generating set.)  For each left coset $gP$ of some $P\in \mc{P}$ there is a copy $\Gamma_{gP}$ of the Cayley graph $\Gamma(P,P\cap S)$ contained in the Cayley graph $\Gamma(G,S)$.
  The \emph{combinatorial cusped space} $X_{CH}(G,\mc{P},S)$ is obtained from $\Gamma(G,S)$ by gluing, to each such coset, a copy of the combinatorial horoball based on $\Gamma(P,P\cap S)$.
\end{definition}

There are many equivalent definitions of relative hyperbolicity (see \cite{Hruska}).  The following definition is from \cite{GM08}.
\begin{definition}
  The pair $(G,\mc{P})$ is \emph{relatively hyperbolic} if some (equivalently any) combinatorial cusped space $X_{CH}(G,\mc{P},S)$ is Gromov hyperbolic.
\end{definition}
\begin{remark}
  The combinatorial cusped space defined in \cite{GM08} also has $2$--cells which make it simply connected.  We will only use the $1$--skeleton in this paper, preferring a different method for obtaining a cusped space with higher connectedness properties (Section \ref{sec:N-connected}).
\end{remark}

\subsubsection{Bowditch boundary}\label{subsub:convergence}
\begin{definition}
  Let $(G,\mc{P})$ be a relatively hyperbolic pair.  The Gromov boundary of the combinatorial cusped space is denoted $\partial(G,\mc{P})$, and called the \emph{Bowditch boundary} of $(G,\mc{P})$.
\end{definition}
The pair $(G,\mc{P})$ acts on $\partial(G,\mc{P})$ as a \emph{geometrically finite convergence group}.  This means the following:
\begin{enumerate}
\item $G\acts\partial(G,\mc{P})$ is \emph{convergence}, meaning $G$ acts properly discontinuously on the set of distinct triples of points in $\partial(G,\mc{P})$.
\item Every $z\in\partial(G,\mc{P})$ is either
  \begin{enumerate}
  \item a \emph{conical limit point}, meaning there is a sequence $\{g_i\}_{i\in\N}$ and a pair of distinct points $a,b\in \partial(G,\mc{P})$ so that $\lim_{i\to\infty}g_i(z) = b$ and $\lim_{i\to\infty}g_i(x) = a$ uniformly for all $x\neq z$, or
  \item a \emph{bounded parabolic point}, meaning that the stabilizer of $z$ acts properly cocompactly on $\partial(G,\mc{P})\setminus\{z\}$.
  \end{enumerate}
\item Each $P\in\mc{P}$ is the stabilizer of some bounded parabolic point, and every bounded parabolic point has stabilizer conjugate to exactly one $P\in\mc{P}$.
\end{enumerate}
The following theorem of Yaman shows the Bowditch boundary is well-defined.

\begin{theorem}\cite{Yaman04}\label{thm:yaman}
  Let $M$ be a nonempty perfect metrizable compactum with a $G$--action.  The following are equivalent:
  \begin{enumerate}
  \item $(G,\mc{P})$ acts as a geometrically finite convergence group on $M$.
  \item $(G,\mc{P})$ is relatively hyperbolic and $M$ is equivariantly homeomorphic to $\partial(G,\mc{P})$. 
  \end{enumerate}
\end{theorem}

Note that $\partial(G,\mc{P})$ is perfect so long as $\mc{P}$ contains no finite group. 

\subsubsection{Finite index subgroups}
\begin{definition}\label{def:induced}
  Let $(G,\mc{P})$ be a group pair.  Suppose that $H<G$ is finite index.  We define an induced peripheral structure $\mc{P}_H$ on $H$.  For each $i$, let $D_i$ be a collection of representatives of double coset space $H\backslash G/P_i$, and define
\[ P_{H} = \{ H\cap dP_id^{-1}\mid d\in D_i, P_i\in \mc{P}\}.\]
\end{definition}

\begin{lemma}\label{lem:finiteindex}
  If $(G,\mc{P})$ is relatively hyperbolic, and $H<G$ is finite index, then $(H,\mc{P}_H)$ is relatively hyperbolic, and $\partial(H,\mc{P}_H)\cong \partial(G,\mc{P})$.
\end{lemma}
\begin{proof}
  The subgroup $H$ acts as a geometrically finite convergence group on $\partial(G,\mc{P})$, and every parabolic fixed point has stabilizer conjugate in $H$ to exactly one $P\in \mc{P}_H$.  Yaman's theorem \ref{thm:yaman} implies that $(H,\mc{P}_H)$ is relatively hyperbolic with Bowditch boundary homeomorphic to $\partial(G,\mc{P})$.
\end{proof}

\subsection{\v{C}ech Cohomology and Singular Cohomology}

\begin{definition}
A space $X$ is \emph{homologically locally connected in dimension $n$} (or $HLC^n$) if, for each $x\in X$ and neighborhood $U$ of $x$, there is a neighborhood $V\subseteq U$ of $x$ such that the induced map $H_i(V;\Z)\rightarrow H_i(U;\Z)$ on reduced homology is trivial for $i\le n$.
\end{definition}

The following proposition is in Spanier \cite[Corollaries 6.8.8 and 6.9.5]{Spanier}.

\begin{prop}\label{prop: Cech Singular Iso}
Suppose $X$ is $HLC^n$, Hausdorff, and paracompact.
Let $A$ be an abelian group.
Then, $\check{H}^i(X;A)\cong H^i(X;A)$ for $i=0,...,n$.
\end{prop}

\subsection{Cohomology of Group Pairs}\label{ss:relative cohomology}

\begin{notation}
When we write $\Hom_G,\op{Ext}_G$ and $\op{Tor}^G$, we take this to mean the $\Hom_{\Z G},\op{Ext}_{\Z G}$ and $\op{Tor}^{\Z G}$, respectively.
We let $A$ denote a (discrete) ring.
The cohomology of a group with coefficients in $M$, $H^k(G;M)$, is $\op{Ext}_G^k(\Z;M)$ and the homology, $H_k(G;M)$, is $\op{Tor}_k^G(\Z;M)$ where $\Z$ has a trivial $G$-action.
\end{notation}

The cohomology of a group pair will be defined following \cite{BE78}.
Let $G$ be a group and let $\mathcal{P}$ be a family of subgroups.
Define the $G$-module $\Z G/\mathcal{P}:=\oplus_{P\in\mathcal{P}}Z[G/P]$ and let $\Delta_{G/\mathcal{P}}$ be the kernel of the augmentation $\Z G/\mathcal{P}\rightarrow\Z$.
Then, for a $G$-module $M$, the relative cohomology groups $H^k(G,\mathcal{P};M)$ are defined to be $\op{Ext}_G^{k-1}(\Delta_{G/\mathcal{P}},M)$.
Similarly, the relative homology groups $H_k(G,\mathcal{P};M)$ are defined to be $\op{Tor}_{k-1}^G(\Delta_{G/\mathcal{P}},M)$.
\begin{remark}
  We recall what this means:  Let $$\cdots\to F_2 \to F_1 \to F_0 \to \Delta_{G/\mathcal{P}}$$ be a free resolution of the $G$-module $\Delta_{G/\mathcal{P}}$.  Then $\op{Tor}_*^G(\Delta_{G/\mathcal{P}},M)$ denotes the homology of $$\cdots\to F_2\otimes_G M \to F_1\otimes_G M \to F_0\otimes_G M,$$ whereas $\op{Ext}_G^*(\Delta_{G/\mathcal{P}},M)$ denotes the (co)homology of $$\Hom_G(F_0,M)\to \Hom_G(F_1,M)\to \Hom_G(F_2,M)\to \cdots .$$

The dimension shift is clarified if one imagines the resolution coming from a contractible simplicial complex $K$ with $G$--action chosen so that the stabilizers of vertices are the conjugates of elements of $\mc{P}$, but that all other cell stabilizers are trivial.  We can then identify $\Z G/\mc{P}$ with $C_0(K)$, and the image of the boundary map $C_1(K)\to C_0(K)$ is then equal to $\Delta_{G/\mc{P}}\subset C_0(K)$.  Now setting $F_i = C_{i+1}(K)$ gives a free resolution of $\Delta_{G/\mc{P}}$.
\end{remark}

Crucially, there are long exact sequences of pairs.
\begin{prop}\label{prop:BEles}\cite[Prop 1.1]{BE78}
  For any group pair $(G,\mc{P})$ and any $G$--module $M$, there are long exact sequences in cohomology and homology:
\[
...\rightarrow H^k(G;M)\rightarrow H^k(\mathcal{P};M)\rightarrow H^{k+1}(G,\mathcal{P};M)\rightarrow H^{k+1}(G;M)\rightarrow...
\]
\[
...\rightarrow H_{k+1}(G;M)\rightarrow H_{k+1}(G,\mathcal{P};M)\rightarrow H_k(\mathcal{P};M)\rightarrow H_k(G;M)\rightarrow...
\]
where $H^k(\mathcal{P};M):=\prod_{P\in\mathcal{P}}H^k(P;M)$ and $H_k(\mathcal{P};M):=\oplus_{P\in\mathcal{P}}H_k(P;M)$.
\end{prop}

Let $K$ be a $K(G,1)$ cell complex and let $\{L_P\}_{P\in\mathcal{P}}$ be disjoint $K(P,1)$ subcomplexes such that each inclusion induces the inclusion $P\hookrightarrow G$ on $\pi_1$ (after a choice of path connecting the base points).
Let $L:=\sqcup_{P\in\mathcal{P}}L_P$.
Then $(K,L)$ is called an \textit{Eilenberg-MacLane pair} for $(G,\mathcal{P})$.
Bieri and Eckmann provide a topological interpretation of the relative cohomology groups.
\begin{theorem}\label{thm: EM pairs compute cohomology}\cite[Thm 1.3]{BE78}
If $(K,L)$ is an Eilenberg-MacLane pair for $(G,\mathcal{P})$ and $M$ is a $G$-module, then there are the following diagrams of long exact sequences
\[
\begin{tikzcd}
...\arrow{r}&H^k(G,\mathcal{P};M)\arrow{d}\arrow{r}&H^k(G;M)\arrow{d}\arrow{r}&H^k(\mathcal{P};M)\arrow{d}\arrow{r}&H^{k+1}(G,\mathcal{P};M)\arrow{d}\arrow{r}&...\\
...\arrow{r}&H^k(K,L;M)\arrow{r}&H^k(K;M)\arrow{r}&H^k(L;M)\arrow{r}&H^{k+1}(K,L;M)\arrow{r}&...
\end{tikzcd}
\]
\[
\begin{tikzcd}
...\arrow{r}&H_k(L;M)\arrow{r}\arrow{d}&H_k(K;M)\arrow{r}\arrow{d}&H_k(K,L;M)\arrow{r}\arrow{d}&H_{k-1}(L;M)\arrow{r}\arrow{d}&...\\
...\arrow{r}&H_k(\mathcal{P};M)\arrow{r}&H_k(G;M)\arrow{r}&H_k(G,\mathcal{P};M)\arrow{r}&H_{k-1}(\mathcal{P};M)\arrow{r}&...
\end{tikzcd}
\]
which are commutative up to sign and where the vertical arrows are isomorphisms.
\end{theorem}

In proving Theorem \ref{thm: EM pairs compute cohomology}, Bieri and Eckmann prove the following statement.

\begin{prop}\label{prop: chain complexes of EM pairs compute cohomology}
Let $(K,L),(G,\mc{P})$ and $M$ be as in \ref{thm: EM pairs compute cohomology}.
Let $\tilde{K}$ be the universal cover of $K$ and let $\tilde{L}$ be the preimage of $L$ under $\tilde{K}\rightarrow K$.
If $C_*(\tilde{K},\tilde{L})$ is the relative chain complex (this can be singular, cellular or simplicial), then the homology of $C_*(\tilde{K},\tilde{L})\otimes_G M$ is $H_*(G,\mc{P};M)$ and the cohomology of $\Hom_G(C_*(\tilde{K},\tilde{L});M)$ is $H^*(G,\mc{P};M)$.
\end{prop}

Following \cite{K09}, we make the following definitions
\begin{definition}
Suppose $(G,\mathcal{P})$ satisfies the following
\begin{itemize}
\item[i)] $G$ and each $P\in\mathcal{P}$ is type $FP$
\item[ii)] $\mathcal{P}$ is a finite collection of subgroups
\end{itemize}
Then the pair $(G,\mathcal{P})$ is \textit{type $FP$}.
\end{definition}
\begin{remark}
In the case that $(G,\mathcal{P})$ is torsion-free relatively hyperbolic and $\mathcal{P}$ consists of type $F$ subgroups, Dahmani shows that $(G,\mathcal{P})$ is type FP \cite{D03}.
\end{remark}
\begin{definition}
The \textit{cohomological dimension} of $(G,\mathcal{P})$ is
\[
\op{cd}(G,\mathcal{P}):=\max\{n\in\N\mid H^n(G,\mathcal{P};M)\neq0 \mbox{ for some } \Z G\text{--module } M\}
\]
\end{definition}
Kapovich observes that the following statement can be proved in the same way as the corresponding absolute statement.
\begin{prop}\cite[Lemma 2.9]{K09}\label{prop:kap}
If $(G,\mathcal{P})$ is type $FP$, then
\[
\op{cd}(G,\mathcal{P})=\max\{n\in\N \mid H^n(G,\mathcal{P};\Z G)\neq0\}
\]
\end{prop}

\section{An $N$--connected cusped space for $(G,\mc{P})$.}\label{sec:N-connected}
The combinatorial cusped space described in \cite{GM08} and the relative Cayley complex described in \cite{Osin} are well-suited for arguments involving $2$--dimensional filling problems, but are not so useful for higher-dimensional homotopy theoretic arguments.  Dahmani shows in \cite{D03} that if the peripheral subgroups of a pair $(G,\mc{P})$ have finite classifying spaces, then these can be extended to give a finite-dimensional classifying space for $G$ provided that $G$ is torsion free.  Moreover he is able to build a $Z$--set compactification of $G$, given such compactifications for the peripherals.  In the present work, we do not want to assume type $F$, but only type $F_\infty$, and moreover we do not want to assume that the peripherals have nice compactifications, so we take different approach. We build, for each $N$, an $N$--connected finite dimensional version of the cusped space $\mc{X}(N)$, for which the Bowditch boundary will form a kind of ``weak $Z$--set compactification''.  Key features of this space are
\begin{enumerate}
  \item The compactly supported $k$--dimensional cohomology $H_c^k(\mc{X}(N);A)$ can be identified with $H^k(G,\mc{P};A G)$ for $k\leq N$ (Proposition \ref{prop:topiso}).
  \item The compactified space $\overline{\mc{X}(N)}=\mc{X}(N)\cup \partial(G,\mc{P})$ has vanishing cohomology up to dimension $N$ (Lemma \ref{lem:vanish}).
  \end{enumerate}
The isomorphism between $H_c^k(\mc{X}(N);A)$ and $\check{H}^{k-1}(\partial(G,\mc{P});A)$ then follows, for $k\leq N$, from the long exact sequence for a pair (see Proposition \ref{prop:boundaryiso}).

Accordingly, we fix $(G,\mc{P})$ an $F_\infty$ group pair with $\mathcal{P}$ a finite family of subgroups, and an integer $N\ge 0$.  We will build a locally compact $N$--connected metric simplicial complex $\mc{X}(G,\mc{P},N)$ which can be used to compute the cohomology of the pair $(G,\mc{P})$ in some range.  Since the group pair will be fixed, we will write $\mc{X}(N)$.  Later we will put a metric on this space so it is quasi-isometric to the cusped space for $(G,\mc{P})$.

\subsection{Topology of the $N$--connected cusped space}

\begin{definition}[The $N$-Connected Cusped Space]\label{def:curlyX}
  By assumption both $G$ and the members of $\mc{P}$ have classifying spaces with finitely many cells in each dimension.  Let $B_G$ and $B_P$ for $P\in \mc{P}$ be (pointed) $(N+1)$--skeleta of such classifying spaces, chosen to have the following extra properties:
  \begin{enumerate}
  \item The universal covers $E_G\to B_G$ and $E_P\to B_P$ for $P\in \mc{P}$ are simplicial complexes.
  \item For each $P\in \mc{P}$, an inclusion $\iota_P\co B_P\to B_G$ induces the inclusion of $P$ into $G$.
  \end{enumerate}

  (The notation $B_G$ and $E_G$ is meant to emphasize that the spaces involved in this construction come from classifying spaces and their universal covers but that $B_G\not\simeq BG$ and $E_G\not\simeq EG$ in general.)
  
  Now let $\mathrm{Cyl}$ be the \emph{open} mapping cylinder of $\iota_{\mc{P}}=\bigsqcup_{\mc{P}}\iota_P\co \bigsqcup_{\mc{P}} B_P \to B_G$, i.e.
  \[\mathrm{Cyl} = \left. B_G \sqcup \left([0,\infty)\times \bigsqcup_{\mc{P}}B_P \right) \middle/ \iota_{\mc{P}}(x)\sim (0,x)\right. .\]
See Figure \ref{fig:cyl}.
\begin{figure}[htbp]
  \centering
  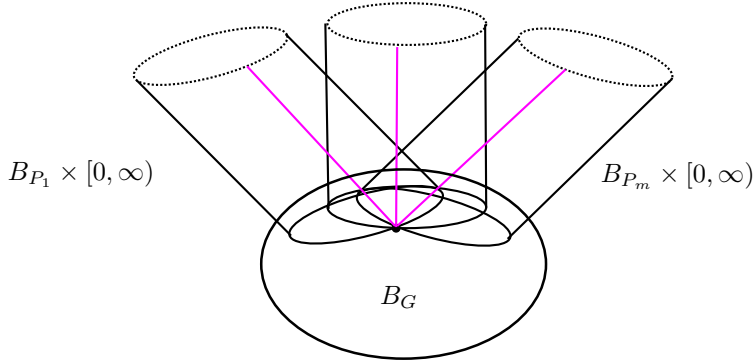
  \caption{The mapping cylinder $\mathrm{Cyl}$.  The purple wedge of rays is $W$ in the proof of Proposition \ref{prop:gromovhyperbolic}.}
  \label{fig:cyl}
\end{figure}
As a topological space, $\mc{X}(N)$ is the universal cover of $\mathrm{Cyl}$.  There is a $G$--equivariant simplicial structure whose $0$--simplices are the $0$--simplices of the universal cover $E_G$ of $B_G$ together with all the points in the preimage of $(v,n)$ where $v$ is a vertex of some $B_P$ and $n\in \Z_{\geq 0}$.  We also assume that the cell structure on each $[0,\infty)\times B_P$ is chosen in some standard way, so that the shift $(t,k)\mapsto (t+1,k)$ gives a simplicial embedding of $[0,\infty)\times \widetilde{B_P}$ into itself.

Equipped with this simplicial structure, $\mc{X}(N)$ will be called the \emph{$N$--connected cusped space}.
\end{definition}
We will put a metric on $\mc{X}(N)$ below in Definition \ref{def:curlyXmetric}.

  \begin{prop}\label{prop:topiso}
    The compactly supported cohomology of $\mc{X}(N)$ is the group cohomology of $(G,\mc{P})$.  More specifically, for $k\leq N$, there are isomorphisms
    \[ H^k(G,\mc{P};A G) \to H_c^k(\mc{X}(N);A).\]
  \end{prop}
  \begin{proof}
    The space $\mc{X}(N)$ is an $(N+2)$--dimensional locally finite complex, so it is locally compact and locally contractible.  Define $E$ to be the full subcomplex on the vertices of depth $\leq 1$, and let $V$ be the full subcomplex on the vertices of depth exactly $1$.

Consider the short exact sequence of simplicial cochains with compact support:
\[
0\rightarrow C_c^*(E,V;A)\rightarrow C_c^*(E;A)\rightarrow C_c^*(V;A)\rightarrow0
\]
Here, $C_c^*(E,V;A)$ is defined to be the kernel of the restriction $C_c^*(E;A)\rightarrow C_c^*(V;A)$.

  By construction, the space $E\setminus V$ is homeomorphic to $\mc{X}(N)$.  By Theorem \ref{thm: rel cochains with support}, the cohomology of $C_c^*(E,V;A)$ is $H_c^*(\mc{X}(N);A)$.
  
  By \cite[Lemma VIII.7.4]{Bro87} the compactly supported simplicial cochain complexes $C_c^*(E;A)$ and $C_c^*(V;A)$ are naturally isomorphic as $AG$ modules to the complexes $\Hom_G(C_*(E),AG)$ and $\Hom_G(C_*(V),AG)$, respectively.  (Brown assumes $A=\Z$ in the statement, but the proof goes through for an arbitrary ring.)
Because the isomorphism is natural, there is a map of $AG$-modules such that the following diagram commutes:
\[
\begin{tikzcd}
0\arrow{r}&C_c^*(E,V;A)\arrow{r}\arrow{d}&C_c^*(E;A)\arrow{r}\arrow{d}&C_c^*(V;A)\arrow{r}\arrow{d}&0\\
0\arrow{r}&\Hom_G(C_*(E,V);AG)\arrow{r}&\Hom_G(C_*(E);AG)\arrow{r}&\Hom_G(C_*(V);AG)\arrow{r}&0
\end{tikzcd}
\]

For $k\le N$, the $k$-th cohomology of $\Hom_G(C_*(E,V),AG)$ is $H^k(G,\mc{P};AG)$ by an application of Proposition \ref{prop: chain complexes of EM pairs compute cohomology}.  Using the Five Lemma, we deduce
\[ H_c^k(\mc{X}(N);A)\cong H^k(G,\mc{P};AG)\mbox{, when } k\leq N .\]
\end{proof}

We record the following consequence of the proof of Proposition \ref{prop:topiso}, which will be used in the proof of Lemma \ref{lem: inclusion of open horoball}.
\begin{addendum}\label{add: part of main thm proof}
With notation as in the proof of \ref{prop:topiso}, there is the following diagram where the vertical arrows are isomorphisms of $AG$-modules.
\[
\begin{tikzcd}
\cdots\arrow{r}&H_c^{N-1}(V;A)\arrow{r}\arrow{d}&H_c^N(E,V;A)\arrow{r}\arrow{d}&H_c^N(E;A)\arrow{r}\arrow{d}&H_c^N(V;A)\arrow{d}\\
\cdots\arrow{r}&H^{N-1}(\mc{P};AG)\arrow{r}&H^N(G,\mc{P};AG)\arrow{r}&H^N(G;AG)\arrow{r}&H^N(\mc{P};AG)
\end{tikzcd}
\]
\end{addendum}
\begin{remark}
  In case $X=\mc{X}(N)$ is contractible and admits a $Z$--set compactification $\overline{X}=X\cup Z$, it is standard to show that $\check{H}^{k-1}(Z;A)\cong H_c^k(X;A)$ for each $k$ (see for instance the proof below of Proposition \ref{prop:boundaryiso}).  In this case we obtain isomorphisms $\check{H}^{k-1}(Z;A)\cong H^k(G,\mc{P};A G)$ for each $k$.  In the current paper we focus on relatively hyperbolic group pairs.  Another family of examples is furnished by pairs $(G,\mc{P})$ where $G$ acts geometrically on a CAT$(0)$ cell complex $E_G$ and each $P$ acts geometrically on a convex subcomplex $E_P$.  In this case the space $\mathrm{Cyl}$ can be given a proper nonpositively curved metric, and its universal cover $X$ can be compactified with its CAT$(0)$ boundary (see \cite[II.8]{BH99}).  
\end{remark}

\subsection{Geometry of the $N$--connected cusped space}

 Next we describe a metric on $\mc{X}(N)$ making it quasi-isometric to the combinatorial cusped space.  The main tool is the warped product, which was first extended to the metric geometry setting by Chen \cite{C99}.

\begin{definition}\label{def: warped product}
Let $X$ and $Y$ be two length spaces and let $f:X\rightarrow[0,\infty)$ be a continuous function.
Let $\gamma:[0,1]\rightarrow X\times Y$ be a path where $\gamma(t)=(\alpha(t),\beta(t))$.
Suppose $\tau=\{0=t_0<...<t_n=1\}$ is a partition of the interval and define
\[
\ell_\tau(\gamma)=\sum_{i=1}^n\left(d_X(\alpha(t_i),\alpha(t_{i-1}))^2+f(\alpha(t_i))^2d_Y(\beta(t_i),\beta(t_{i-1}))^2\right)^{\frac{1}{2}}
\]
The length of $\gamma$ is defined to be the supremum of $\ell_\tau(\gamma)$ over all partitions $\tau$.
This gives a pseudometric on $X\times Y$ and, if $f$ has no zeros, a metric.
The resulting space with this pseudometric is the \textit{warped product of $X$ and $Y$ with respect to $f$} and is denoted as $X\times_fY$.
\end{definition}

\begin{definition}\label{def:curlyXmetric}
  We describe a path metric on  $\mc{X}(N)$ by putting path metrics on various subsets.  Simplices in the universal cover $E_G$ of $B_G$ are metrized as regular euclidean simplices with unit edge lengths.  Let $H$ be a component of the preimage of some $[0,\infty)\times B_P$ (henceforth called a \emph{horoball}).
  For each $P$, the universal cover $E_P$ of $B_P$ inherits a path metric from $E_G$, and we use this path metric to metrize $H$ as a warped product
  \[ H = [0,\infty)\times_{2^{-t}} E_P .\]
  Henceforth when we refer to $\mc{X}(N)$ or the $N$--connected cusped space, we will assume it has been given this metric.
\end{definition}

\begin{remark}
  The proper space studied by Bowditch in \cite{Bowditch12} can be recovered as a special case when the complexes $B_G$ and $B_P$ each have a single $0$-cell, if we take $N=0$ and replace the warping function $2^{-t}$ with $e^{-t}$.  In fact the exact exponential warping function is not important to the quasi-isometry type, and by \cite[Proposition A.5]{GMS}  this space is always equivariantly quasi-isometric to the combinatorial cusped space from Definition \ref{def:combinatorial cusped}.
\end{remark}

\begin{definition}\label{def: depth}
Let $x\in\mathrm{Cyl}$.
If $x$ can be identified with a point $(t,y)\in[0,\infty)\times B_P$ then we define the \emph{depth} of $x$ (denoted $\text{Depth}(x)$) to be $t$.
Otherwise, $\text{Depth}(x)=0$.

If $\tilde{x}\in\mc{X}(N)$, then $\text{Depth}(\tilde{x}):=\text{Depth}(x)$ where $x$ is the image of $\tilde{x}$ under $\mc{X}(N)\rightarrow\mathrm{Cyl}$.
\end{definition}

The following is immediate from the construction.
\begin{lemma}\label{lem:bdddiam}
  Each simplex of $\mc{X}(N)$ has bounded diameter, and there is a $C>0$ so that the $0$--skeleton is $C$--dense.  In particular nearest-point projection to the $0$--skeleton is a quasi-isometry.
\end{lemma}

The main result of this subsection is the following.
\begin{prop}\label{prop:gromovhyperbolic}
  $\mc{X}(N)$ is quasi-isometric to $X_{CH}$.  In particular, $\mc{X}(N)$ is Gromov hyperbolic if and only if $(G,\mc{P})$ is relatively hyperbolic.
\end{prop}
\begin{proof}
  By Lemma \ref{lem:bdddiam}, it suffices to show that there is a quasi-isometry between $\mc{X}(N)^{(0)}$ and $X_{CH}^{(0)}$.  We show there are coarsely Lipschitz quasi-inverse maps between $\mc{X}(N)^{(0)}$ and the $0$--skeleton
  \[ X_{CH}^{(0)} = G \sqcup \bigsqcup_{\mc{P}}\bigsqcup_{gP\in G/P}\mathbb{Z}_{>0}\times gP\]
    of the combinatorial cusped space.  It follows easily that these maps are quasi-isometries.  We then conclude using the quasi-isometric invariance of Gromov hyperbolicity.

    We first define $\iota \co X_{CH}^{(0)}\to \mc{X}(N)^{(0)}$.  Using the notation in Definition \ref{def:curlyX}, let $W\subset\mathrm{Cyl}$ be a wedge of rays centered at the basepoint $b_G$ of $B_G$, so each ray is equal to the ray $[0,\infty)\times b_P$ inside $\mathrm{Cyl}$.  Let $\tilde{W}$ be a lift to the universal cover $\mc{X}(N)$, and let $\tilde{b}_G$ and $\widetilde{(n,b_P)}$ be the corresponding vertices of this lift.  Define $\iota(g) = g\tilde{b}_G$.  For $gp\in gP$ and $n>0$ define $\iota( (n,gp) ) = gp (\widetilde{(n,b_P)}).$

    The map $\iota$ is injective with image in $\mc{X}(N)^{(0)}$.  Define $\pi\co \mc{X}(N)^{(0)}\to X_{CH}^{(0)}$ by $\pi(x) = \iota^{-1}(\bar{x})$, where $\bar{x}$ is some closest point to $x$ in the image of $\iota$.  Obviously $\pi\circ\iota$ is the identity.  Moreover, the image of $\iota$ is $K$--dense for some $K$, so $\iota\circ\pi$ is within $K$ of the identity.  It is also easy to see that $\iota$ is $K$--Lipschitz for some $K$.

    We now show that $\pi$ is coarsely Lipschitz.  Since the image of $\iota$ is $K$--dense, a standard argument shows that, if we can find a bound on the diameter of $\iota^{-1}(B_{3K}(p))$ independent of $p\in \mc{X}(N)$, then $\pi$ is coarsely Lipschitz.

    (Here is the standard argument: Given $p,q\in \mc{X}(N)$, choose points $p_0,\ldots,p_n$ on a geodesic from $p$ to $q$ so that $p_0=p_i$, $p_n=q$, and $d(p_i,p_{i+1})=K$, except that $d(p_{n-1},p_n)\leq K$.  Then we have $d(p,q)\geq (n-1)K$.  Now choose points $b_i = \iota(a_i)$ so that $d(b_i,p_i)\leq K$, and $a_0=\pi(p)$, $a_n = \pi(q)$.  Now we estimate
\[ d(\pi(p),\pi(q)) \leq \sum_{i=1}^nd(a_{i-1},a_i)\leq n R \leq \frac{R}{K} d(p,q) + R, \]
where $R$ is the bound on the diameter of $\iota^{-1}(B_{3K}(p))$.)

    Let $E_G\subset \mc{X}(N)$ be the universal cover of the $(N+1)$--skeleton of the classifying space $B_G$.  The $G$--action is cocompact in any closed equivariant neighborhood of $E_G$, so there is some constant $B_1$ bounding the diameter of $\iota^{-1}(B_{3K}(p))$ for any $p$ in the $12K$--neighborhood of $E_G$.

    Let $B$ be a $3K$--ball in $\mc{X}(N)$ whose center is in a horoball $H$, at depth at least $12K$.  Suppose $H$ is stabilized by $P^g$, where $P\in \mc{P}$.  Then $H$ is isometric to $[0,\infty)\times_{2^{-t}}E_P$, where $E_P$ is the universal cover of the $(N+1)$--skeleton of $B_P$, a classifying space for $P$.  

    Let $(n,x)$ and $(m,y)$ be points of $B$ in the image of $\iota$, and let $\gamma$ be a geodesic joining them.  We may suppose $n\leq m$, and note that $m-n\leq 6K$.
By our assumptions, the geodesic $\gamma$ lies entirely in $H$, and can be written in terms of the product structure as $(\gamma_1,\gamma_2)$, where $\gamma_2$ is a geodesic in $E_P$.  Because of the warping of the metric, we have
\[ 2^ml(\gamma)\geq  l(\gamma_2). \]
Since $\gamma$ has length at most $6K$, we get (writing $d_{E_P}$ for the path metric on $E_P$)
\[ d_{E_P}((0,x),(0,y))\leq 6 K 2^m.\]
    Note that $(0,x)$ and $(0,y)$ are in the image of $\iota(gP)$.  Let $\Gamma_{gP}$ be the copy of the Cayley graph of $P$ spanned by the vertices $gP$ in $X_{CH}$.  Then $\pi|_{E_P}\co E_P\to \Gamma_{gP}$ is a $(\lambda,\epsilon)$--quasi-isometry for some $\lambda\geq 1$ and $\epsilon>0$ depending only on $P\in \mc{P}$.  We thus have
\[ d_{\Gamma_{gP}}(\pi((0,x)),\pi((0,y)))\leq {6K}\lambda 2^m+\epsilon . \]
It follows that 
\[ d_{X_{CH}}((\pi((n,x)),\pi((n,y))))\leq 2^{m-n}(6K\lambda) +2^{-n}\epsilon +1\leq 2^{6K}(6K \lambda)+\epsilon +1, \]
and finally that
\[ d_{X_{CH}}(\pi((n,x)),\pi((m,y)))\leq 2^{6K}(6K\lambda)+\epsilon +1 + 6K. \]

   The constants $\lambda$ and $\epsilon$ depended on $P$, but there are only finitely many possibilities, so taking the maximum gives us a universal bound $B_2$ on the diameter of $\iota^{-1}(B)$ where $B$ is a ${3K}$--ball whose center is at depth at least $12K$.  Taking $\max\{B_1,B_2\}$ gives the desired universal bound for all ${3K}$--balls.  
\end{proof}
\begin{definition}
  The quasi-isometry from Proposition \ref{prop:gromovhyperbolic} gives an identification of $\partial{X_{CH}}$ with $\partial(G,\mc{P})$.  We use this identification and write $\overline{\mc{X}(N)} = \mc{X}(N)\cup \partial(G,\mc{P})$.
\end{definition}

\subsection{Collapsing spheres near infinity}
The space $\mc{X}(N)$ is Gromov hyperbolic (Proposition \ref{prop:gromovhyperbolic}) and the $0$--skeleton is $C$--dense (Lemma \ref{lem:bdddiam}).  We can therefore fix some $D\geq 1$ so that the conclusion of Lemma \ref{lem:contract nearby} holds, which means roughly that subcomplexes of the Rips complex $R=R_D(\mc{X}(N)^{(0)})$ can be contracted in their ``convex hulls''.

\begin{definition}\label{def:depth preserving map}
A continuous map $r:R_D^{(N+1)}\rightarrow\mc{X}(N)$ is \emph{depth preserving} if, for each $\sigma$ a simplex of $R_D$ and $I\subset[0,\infty)$ the smallest interval containing $\op{Depth}(\sigma^{0})$,
\[
\op{Depth}(r(\sigma))\subset\begin{cases}
[0,\sup I]&\inf I\le D\\
I&\inf I>D
\end{cases}
\]
\end{definition}  

\begin{lemma}\label{lem:fake retraction}
    There are equivariant proper maps $r: R^{(N+1)}\to \mc{X}(N)$ and  $\iota_k : \mc{X}(N)^{(k)}\to R^{(k)}$ for each $k=0,...,N+2$ satisfying the following:
	\begin{enumerate}
    \item $r$ is depth-preserving.
    \item If $k\leq N+1$, then $ r \circ \iota_k$ is the inclusion $\mc{X}(N)^{(k)}\subseteq\mc{X}(N)$.
    \end{enumerate}
\end{lemma}
\begin{proof}
The vertices of $R$ can be identified with the vertices of $\mc{X}(N)$.  In particular, since $D\geq 1$, any simplex of $\mc{X}(N)$ corresponds to a simplex of $R^{(N+2)}$ with the same vertices.  This correspondence gives us the inclusion $\iota_k\co \mc{X}(N)^{(k)}\to R^{(k)}$.

We will construct $r$ inductively. We use the identification already mentioned to define $r$ on $R^{(0)}$. Suppose $r$ has been extended to the a depth preserving map on the $j$-skeleton. Then, let $\sigma\subseteq R$ be an orbit representative $(j+1)$-simplex.  If $\sigma$ is present in $\mc{X}(N)$, we define $r$ to be the inverse of $\iota$ on $\sigma$.

Otherwise, let $I=[a,b]$ be as in the definition above. Suppose first that $b>D$. Then, all of the vertices must be in the same horoball and $\op{Depth}(r(\partial\sigma))\subseteq[a,b]$. So $r(\partial\sigma)$ is in a horoball. In particular, $r(\partial\sigma)$ is in the product of an $N$-connected space with $[a,b]$. Since $j\le N$, we can extend this map to $\sigma$. Suppose that $b\le D$. Then, $\op{Depth}(r(\partial\sigma))\le D$. But the points of depth at most $D$ is an $N$-connected space. This allows us to extend $r$ to $\sigma$ as desired.  We extend equivariantly to the simplices in the orbit of $\sigma$.
\end{proof}

The map $r$ from Lemma \ref{lem:fake retraction} will not be a quasi-isometry in general, but the fact that it is depth preserving will allow us to extend it continuously to the boundary.

\begin{lemma}\label{lem: sequences of unbounded depth in horoballs}
	Let $\{x_j\}_{j\in\N}$ and $\{y_j\}_{j\in\N}$ be sequences tending to infinity in $\mc{X}(N)$ such that
      \begin{enumerate}
      \item $\lim_{j\to\infty}\min\{\text{Depth}(x_j),\text{Depth}(y_j)\} = \infty$;
      \item for each $j$, $x_j$ and $y_j$ lie in the same horoball; and
      \end{enumerate}
  Then $\{x_j\}_{j\in\N}$ and $\{y_j\}_{j\in\N}$ have the same limit point in $\partial(G,\mc{P})$
\end{lemma}
\begin{proof}
  Fix a basepoint $e\in\mc{X}(N)$ at depth $0$.  For $j\in\N$, let $H_j$ be the horoball containing $x_j$ and $y_j$.  If a horoball $H$ occurs infinitely often, then the common limit of the two sequences must be the horoball center.

  Otherwise, we can pass to subseqences so that $d(e,H_j)$ is strictly increasing with $j$.  The uniform quasi-convexity of horoballs implies that, for large $j$, any geodesic joining $x_j$ to $y_j$ is contained in $H_j$.

  Suppose for a contradiction that the two sequences do not converge to the same point at infinity.  Then there are indices $i_k,j_k\to\infty$ so that the Gromov products $\gprod{x_{i_k}}{y_{j_k}}{e}$ are bounded.  If $\gamma_k$ is a geodesic joining $x_{i_k}$ to $y_{j_k}$, then $d(e,\gamma_k)$ is likewise bounded.  Using thin triangles, it follows that if $\sigma_k$ is a geodesic joining $x_{i_k}$ to $x_{j_k}$, then $d(e,\sigma_k)$ is bounded, as are the Gromov products $\gprod{x_{i_k}}{x_{j_k}}{e}$.  But this contradicts the hypothesis that $\{x_j\}_{j\in\N}$ tends to infinity.
\end{proof}
The compactification of the Rips complex described in Definition \ref{def:compactification of rips} also gives a compactification of the $(N+1)$--skeleton of the Rips complex.
\begin{prop}\label{prop: r extends to compactification}
	The map $r:R^{(N+1)}\rightarrow\mathcal{X}(N)$ extends to a continuous map $\overline{R^{(N+1)}}\rightarrow\overline{\mathcal{X}(N)}$.  This restricts to the identity on $\partial(G,\mathcal{P})$.
\end{prop}
\begin{proof}
	It suffices to show that if $\{a_j\}_{j\in\N}$ is a sequence of points in $R^{(N+1)}$ limiting to $z\in\partial(G,\mathcal{P})$ then $\{r(a_j)\}$ also limits to $z$.
	
	For each $j$, let $v_j$ denote a vertex of a simplex containing $a_j$. Suppose first that the $r(a_j)$ have bounded depth. Then, $d(r(a_j),r(v_j))$ is bounded which implies $r(a_j)$ approaches $z$.

	Suppose the $r(a_j)$ have unbounded depth. For all but finitely many $j$, $r(a_j)$ and $v_j$ will be in the same horoball so Lemma \ref{lem: sequences of unbounded depth in horoballs} implies that $r(a_j)$ and $v_j$ converge to the same boundary point.
\end{proof}

  Now we prove the main proposition of the section.
  \begin{prop}\label{prop:conditioniii}
    For each $i=0,..., N$, every $z\in\partial(G,\mc{P})$ and every neighborhood $U$ of $z$ in $\overline{\mc{X}(N)}$, there is a neighborhood $V\subseteq U$ of $z$ such that every map $S^i\rightarrow V\setminus\partial(G,\mc{P})$ is nullhomotopic in $U\setminus\partial(G,\mc{P})$.
  \end{prop}
  \begin{proof}
     Given a neighborhood $U\subseteq\overline{\mathcal{X}(N)}$ of $z$, let $v(U)$ be the set of vertices whose closed stars are contained in $U$, let $U_1$ be the interior of the full subcomplex on the vertices $v(U)$, and let $\tilde{U}_1 = U_1\cup\left(U\cap \partial(G,\mc{P})\right)$.  Then we claim $\tilde{U}_1\subset U$ is still an open neighborhood of $z$ in $\mc{X}(N)$.  The intersection with $\mc{X}(N)$ is open by construction so we just need to check $\tilde{U}_1$ contains an open neighborhood of any $w\in U\cap \partial(G,\mc{P})$.  Suppose there is a sequence $\{x_i\}_{i\in \N}$ of vertices of $\mc{X}(N)\setminus U_1$ which converges to $w\in U\cap\partial(G,\mc{P})$.  All but finitely many of these vertices is in $U\setminus U_1$, so all but finitely many of these vertices have closed stars which meet $\mc{X}(N)\setminus U$.  In particular there is a sequence of vertices $\{x_i'\}_{i\in \N}$ outside $U$, but converging to $w$, contradicting our choice of $U$. 

     Let $r$ be as in Lemma \ref{lem:fake retraction} and let $\iota=\iota_{N+2}\co\mc{X}(N)\rightarrow R^{(N+2)}$ be the inclusion of $\mc{X}(N)$ as a subcomplex of the Rips complex.

     Then $r^{-1}(\tilde{U}_1)$ is an open subset of $\overline{R^{(N+1)}}$.  By Proposition \ref{prop:condition iii for Rips complex}, there is a neighborhood $W\subseteq r^{-1}(U_1)$ of $z$ such that spheres of dimension up to $N$ in $W\setminus\partial(G,\mathcal{P})$  are contractible in $r^{-1}(U_1)\setminus\partial(G,\mathcal{P})$.

     Let $v(W)$ be the set of vertices of $W$ whose closed stars in $R^{(N+1)}$ are contained in $W$.  Let $W_1$ be the interior of the full subcomplex of $R^{(N+2)}$ on the vertices $v(W)$, and let $\tilde{W}_1 = W_1\cup \left(W\cap \partial(G,\mc{P})\right)$.  A similar argument to that in the first paragraph shows that $\tilde{W}_1$ is an open neighborhood of $z$ in $R^{(N+2)}$.
Let $V' = \iota^{-1}(W_1)$, let $V_1$ be the interior of the subcomplex of $\mc{X}(N)$ made of cells whose closed stars are in $V'$, and let $V = V_1\cup \left(W\cap\partial(G,\mc{P})\right)$.  Arguing again as in the first paragraph, $V$ is still an open neighborhood of $z$ in $\mc{X}(N)$.

     We establish that $V\subseteq U$.  Indeed, $U$ contains the full subcomplex in $\mc{X}(N)$ spanned by $W\cap \mc{X}(N)^{(0)}$.  This subcomplex contains the intersection of  $W_1$ with  $\mc{X}(N)$ in $R^{(N+2)}$, which is $V'$.  Thus $V\subseteq U$.

     Next we show that maps from $S^i$ into $V\setminus\partial(G,\mc{P})$ are null-homotopic in $U\setminus\partial(G,\mc{P})$.  Let $\alpha\co S^i\to V\setminus\partial(G,\mc{P})$.  We can homotope $\alpha$ in $V$ to have image in the $i$--skeleton of $V' = \iota^{-1}(W_1)$.  Let $\iota_{N+1}\co \mc{X}(N)^{(N+1)}\to R^{(N+1)}$ be as in Lemma \ref{lem:fake retraction}.  Then $\iota_{N+1}\circ\alpha$ has image inside $W$.  Thus there is a homotopy $h_t$ from $\iota_{N+1}\circ\alpha$ to a constant, so that the homotopy occurs completely inside $r^{-1}(U_1)$.  Applying $r$ to the homotopy, we get a homotopy $r\circ h_t$ from $\alpha$ to a constant occurring entirely inside $U_1\subseteq U\setminus\partial(G,\mc{P})$.  
  \end{proof}

\subsection{A $Z$--set when $(G,\mc{P})$ is type $F$}
  
In this section we show that if $(G,\mc{P})$ is type $F$, then $\partial(G,\mc{P})$ gives an equivariant $Z$--set compactification of $\mc{X}(N)$ for large $N$.  This fact isn't needed for the proof of Theorem \ref{thm:maintheorem}, but will be used in Sections \ref{sec:pdpair} and \ref{sec:dimension}.

We observe first that, assuming that $(G,\mc{P})$ has type $F$, we may choose $E_G$ and each $E_P$ in \ref{def:curlyX} to be finite complexes, so that, for some $N$ and all $i\geq N$,  $\mc{X}(i)=\mc{X}(N)$. It follows that this space is contractible.
\begin{theorem}\label{thm:zset}
  If $(G,\mc{P})$ is relatively hyperbolic and type $F$, and $\mc{X}(N)$ is chosen as above, then $\partial(G,\mc{P})$ is a $Z$--set in $\overline{\mc{X}(N)}$.
\end{theorem}

In general, it is difficult to verify that a closed subset is a $Z$--set. To do this, we will use the following from \cite{BM91}.

\begin{prop}\label{prop:BM 2.1}\cite[Proposition 2.1]{BM91}
Suppose $X$ is compact metrizable and that $F\subseteq X$ is closed such that the following conditions are satisfied.
\begin{enumerate}
\item $F$ has empty interior in X
\item $\dim X=n<\infty$
\item For each $i=0,1,...,n$, every $z\in F$ and every neighborhood $U$ of $z$, there is a neighborhood $V\subseteq U$ of $z$ such that a map $S^i\rightarrow V\setminus F$ is nullhomotopic in $U\setminus F$
\item $X\setminus F$ is an ANR
\end{enumerate}
Then, $X$ is an ANR and $F\subseteq X$ is a $Z$-set.
\end{prop}

Before we prove Theorem \ref{thm:zset} we need the following consequence of a result of Dahmani.
\begin{lemma}\label{lem:dahmanilemma}
  For any relatively hyperbolic pair $(G,\mc{P})$, the dimension of $\partial(G,\mc{P})$ is finite.
\end{lemma}
\begin{proof}
  The proof of \cite[Lemma 3.7]{D03} shows that the Gromov boundary of the coned-off Cayley graph is finite dimensional. Since $\partial(G,\mc{P})$ is the union of the boundary of the coned-off Cayley graph with a countable set, $\partial(G,\mc{P})$ is finite dimensional.
\end{proof}

We now prove Theorem \ref{thm:zset}.
\begin{proof}
  Let $n = \dim(\mc{X}(N))\leq N+2$.
  We verify that the hypotheses of Proposition \ref{prop:BM 2.1} hold for $\partial(G,\mc{P})\subseteq\overline{\mc{X}(N)}$. The first, second, and fourth conditions are clear.  Lemma \ref{lem:dahmanilemma} shows that $\partial(G,\mc{P})$ is finite dimensional, so $\overline{\mc{X}(N)}$ is also finite dimensional.  The third condition follows from Proposition \ref{prop:conditioniii} (applied to $\mc{X}(N+2)=\mc{X}(N)$) and the fact that $\mc{X}(N)$ is contractible.
\end{proof}

\subsection{A $Z_{N-1}$--set otherwise}
We return to the setting in which $(G,\mc{P})$ is assumed only to be $F_\infty$ and not type $F$.
In the papers \cite{GS73,GS74}, Geoghegan and Summerhill propose the notion of a $Z_k$--set.
\begin{definition}
  A closed subset $F$ of a space $X$ is a \emph{$Z_k$--set} if, for every nonempty $k$--connected open $U\subseteq X$, the set $U\setminus F$ is also nonempty and $k$--connected.
\end{definition}

Bestvina and Mess's proof of Proposition \ref{prop:BM 2.1} in \cite{BM91} gives the following weaker result that will be important when working with $F_\infty$ groups.
\begin{prop}\label{prop:weak BM 2.1}
Suppose $X$ is compact metrizable and $F\subseteq X$ is closed. Suppose conditions (1) and (4) of Proposition \ref{prop:BM 2.1} are satisfied and that condition (3) holds for some $n$. Then,
\begin{enumerate}
\item \cite[Lemma 2.4]{BM91} Let $P$ be a finite simplicial complex of dimension $\le n$, and let $f\co P\to X$ be a map.  Then there is a homotopy $h\co P\times I\to X$ so that $h(-,0)=f$ and $h(-,t)$ has image in $X\setminus F$ for all $t>0$.
\item \cite[Lemma 2.5]{BM91} For each $i=0,1,...,n$, each $z\in F$ and each neighborhood $U$ of $z$, there is a neighborhood $V\subseteq U$ of $z$ such that every map $S^i\rightarrow V$ is nullhomotopic in $U$.
\end{enumerate}
\end{prop}
The second part of Proposition \ref{prop:weak BM 2.1} differs from Proposition \ref{prop:conditioniii} in that the spheres under consideration are allowed to meet the boundary.

The first part of Proposition \ref{prop:weak BM 2.1} implies that $F\subseteq X$ is a $Z_{n-1}$--set.  Indeed, suppose that $U$ is $(n-1)$--connected, and that $f_0\co S^i\to U\setminus F$ is any map, where $i\leq (n-1)$.  Take $P=S^i\times I$, and let $f\co P\to U$ be a null-homotopy of $f$.   Now apply the first part of Proposition \ref{prop:weak BM 2.1} to obtain $h$.
For some small $\epsilon$, $h|_{P\times(0,\epsilon]}$ has image contained in $U\setminus F$, and thus gives a null-homotopy of $P$ in $U$ which misses $F$.

Using Proposition \ref{prop:conditioniii} and Lemma \ref{lem:dahmanilemma} as in the proof of Theorem \ref{thm:zset} we obtain the following.
\begin{cor}\label{cor:weak Zset}
  The Bowditch boundary $\partial(G,\mc{P})$ is a $Z_{N-1}$--set in $\overline{\mc{X}(N)} = \mc{X}(N)\cup \partial(G,\mc{P})$.
\end{cor}
\begin{remark}
  The astute reader may have noticed that the construction and properties of $\mc{X}(N)$ are the same if $(G,\mc{P})$ is only type $F_{N+1}$ rather than type $F_\infty$.  In particular, Corollary \ref{cor:weak Zset} holds under that weaker hypothesis.

  Similarly, the arguments in Subsections \ref{ss:vanishing} and \ref{ss:mainthm} will give the isomorphisms as in Theorem \ref{thm:maintheorem}, so long as $k\leq N$ and $(G,\mc{P})$ is type $F_{N+1}$ and relatively hyperbolic.
\end{remark}
\subsection{Vanishing of \v{C}ech cohomology of the compactification}\label{ss:vanishing}

\begin{lemma}\label{lem:HLC lemma}
Suppose $V\rightarrow U$ factors as 
\[
V=V_n\rightarrow V_{n-1}\rightarrow\cdots\rightarrow V_0\rightarrow U
\]
where each map induces the trivial homomorphism on $\pi_i$ for $i\le n$.
Then, the induced homomorphism $H_i(V;\Z)\rightarrow H_i(U;\Z)$ is trivial for $i\le n$.
\end{lemma}

\begin{proof}
The map $V_n\rightarrow V_{n-1}$ factors through a connected space, which we will denote $W_0$.
The map $W_0\rightarrow V_{n-2}$ factors through a simply connected space $W_1$.
Proceeding inductively, we see that $V_n\rightarrow U$ factors through an $n$-connected space $W_n$, so the map is trivial on homology.
\end{proof}

The following is the key lemma:
\begin{lemma}\label{lem:vanish}
  For $k\leq N$, $\check{H}^k(\overline{\mc{X}(N)};A)\cong 0$ where the left hand side is reduced \v{C}ech cohomology.
\end{lemma}
\begin{proof}
	We claim that $\overline{\mc{X}(N)}$ is $HLC^N$.
For this, we only need to consider the points on $\partial(G,\mc{P})$. Let $z\in\partial(G,\mc{P})$ and let $U$ be an open neighborhood of $z$ in $\overline{\mc{X}(N)}$.  We need an open neighborhood $V\subseteq U$ of $z$ such that $H_i(V;\Z)\rightarrow H_i(U;\Z)$ is trivial for $i=0,...,N$. By Propositions \ref{prop:conditioniii} and \ref{prop:weak BM 2.1} we see that there is a neighborhood $V_0\subseteq U$ of $z$ such that maps $S^i\rightarrow V_0$ are nullhomotopic in $U$ for $i=0,...,N$. Inductively, we can find $V_j\subseteq V_{j-1}$ such that maps $S^i\rightarrow V_j$ are nullhomotopic in $V_{j-1}$. Applying Lemma \ref{lem:HLC lemma}, we see that $\overline{\mc{X}(N)}$ is $HLC^N$.

By \ref{prop: Cech Singular Iso}, there is the isomorphism $\check{H}^i(\overline{\mc{X}(N)};A)\cong H^i(\overline{\mc{X}(N)};A)$ between \v{C}ech cohomology and singular cohomology for $i\le N$. Now, we show that $\overline{\mc{X}(N)}$ is $N$-connected. Consider a map $f:S^i\rightarrow\overline{\mc{X}(N)}$.  By Proposition \ref{prop:weak BM 2.1}, we may assume that $f(S^i)\cap\partial(G,\mc{P})=\emptyset$. Then $f$ is nullhomotopic because $\mc{X}(N)$ is $N$-connected. So $H_i(\overline{\mc{X}(N)};\Z)\cong0$ for $i\le N$ and, by the universal coefficients theorem, $H^i(\overline{\mc{X}(N)};A)\cong0$ for $0<i\le N$.
Therefore, $\Check{H}^i(\overline{\mc{X}(N)};A)\cong0$ for $0<i\le N$. The $i=0$ case is trivial since we are using reduced \v{C}ech cohomology.
\end{proof}

We can now relate the compactly supported cohomology of $\mc{X}(N)$ with the \v{C}ech cohomology of the Bowditch boundary.
\begin{prop}\label{prop:boundaryiso}
  For $k\leq N$ there is an isomorphism of $AG$-modules
\[ H_c^k(\mc{X}(N);A) \to \check{H}^{k-1}(\partial(G,\mc{P});A).\]
\end{prop}
\begin{proof}
  Since $\overline{\mc{X}(N)}=\mc{X}(N)\cup \partial(G,\mc{P})$ is compact Hausdorff, and $\partial(G,\mc{P})$ is closed in $\overline{\mc{X}(N)}$, \cite[II.10.3]{Bredon} gives the following long exact sequence of sheaf cohomology groups.
\[
\cdots\rightarrow H^{k-1}(\overline{\mc{X}(N)};A)\rightarrow H^{k-1}(\partial(G,\mc{P});A)\rightarrow H^k_c(\mc{X}(N);A)\rightarrow H^k(\overline{\mc{X}(N)};A)\rightarrow\cdots
\]
The space $\mc{X}(N)$ is a CW-complex, so its compactly supported sheaf cohomology is isomorphic to its compactly supported singular cohomology (see Appendix, Proposition \ref{prop: S_0 acyclic}). Additionally, sheaf cohomology is isomorphic to \v{C}ech cohomology for the spaces above.
The result follows from Lemma \ref{lem:vanish}.
\end{proof}
\subsection{Proof of Theorem \ref{thm:maintheorem}}\label{ss:mainthm}
Recall the statement.
\maintheorem*

\begin{proof}  We fix a $k$ and find an isomorphism as in \eqref{mainisomorphism}.  Fix some $N$ so that $k\leq N$ and consider the $N$--connected space $\mc{X}(N)$ defined in Definition \ref{def:curlyX}.  Proposition \ref{prop:topiso} gives the isomorphism $H^k(G,\mc{P};A G) \to H_c^k(\mc{X}(N);A)$.

  Proposition \ref{prop:boundaryiso} then gives the isomorphism
  $ H_c^k(\mc{X}(N);A)\to \check{H}^{k-1}(\partial(G,\mc{P});A)$.
\end{proof}

  We now prove a corollary alluded to in the introduction.
\begin{cor} \label{cor:highdcohomisperipheral}
If $(G,\mathcal{P})$ is an $F_\infty$ relatively hyperbolic group pair, then there is an $N>0$ such that, for all $k\ge N$,
\[
H^k(G;\Z G)\cong H^k(\mathcal{P};\Z G)\cong\oplus_{P\in\mathcal{P}}H^k(P;\Z P)\otimes_{\Z P}\Z G
\]
\end{cor}
\begin{proof}
  By Lemma \ref{lem:dahmanilemma}, the dimension of $\partial(G,\mc{P})$ is finite.  Let $N=\dim\partial(G,\mc{P})+1$, so that $H^k(G,\mc{P};\Z G)\cong \check{H}^{k-1}(\partial(G,\mc{P});\Z)$ vanishes for $k\geq N$.
The first isomorphism then follows from the long exact cohomology sequence of a group pair.
The second isomorphism follows from  \cite[Exercise VIII.5.4a]{Bro87}.
\end{proof}

\section{Boundaries of \PD{n} pairs.}\label{sec:pdpair}
A \PD{n} group pair $(G,\mc{P})$ is a pair which is FP and for which \[H^k(G,\mc{P};\Z G)\cong
\begin{cases}
  \tilde{\Z} & k=n\\
  0 & k\neq n
\end{cases}\]
where $\tilde{\Z}$ is the abelian group $\Z$ with a possibly nontrivial action of $G$.

A thorough discussion of \PD{n} pairs can be found in \cite{BE78}.  The following two results are easily deduced from Theorems 2.1, 4.2, and 6.2 of that paper.

\begin{theorem}\label{thm: pd n isomorphism}
For a \PD{n} pair $(G,\mc{P})$ and a $\Z G$-module $M$, there is a commutative ladder between long exact sequences where the vertical arrows are isomorphisms:
\[
\begin{tikzcd}
\cdots\arrow{r}&H^k(G;M)\arrow{r}\arrow{d}&H^k(\mc{P};M)\arrow{r}\arrow{d}&H^{k+1}(G,\mc{P};M)\arrow{r}\arrow{d}&\cdots\\
\cdots\arrow{r}&H_{n-k}(G,\mc{P};\tilde{\Z}\otimes M)\arrow{r}&H_{n-k-1}(\mc{P};\tilde{\Z}\otimes M)\arrow{r}&H_{n-k-1}(G;\tilde{\Z}\otimes M)\arrow{r}&\cdots
\end{tikzcd}
\]
\end{theorem}

\begin{prop}\label{prop: subgroups of pd n pair}
If $(G,\mc{P})$ is a \PD{n} pair and $P\in\mc{P}$, then $P$ is a \PD{n-1} group.
In particular,
\[
H^k(P;\Z P)\cong\begin{cases}\tilde{\Z}&k=n-1\\0&k\neq n-1\end{cases}
\]
\end{prop}

The main result of this section is the following.
\pdpair*

  Since  $2$-dimensional homology manifolds are manifolds \cite[IX.5.6]{Wilder}, we recover the following result of Tshishiku and Walsh \cite{TshWa}.
\begin{cor}\label{cor:2sphere}
  Suppose $(G,\mathcal{P})$ is type $F$ and relatively hyperbolic.  
  The pair $(G,\mc{P})$ is \PD{3} if and only if $\partial(G,\mathcal{P})\cong S^2$.
\end{cor}

\begin{remark}
In \cite{B96}, Bestvina considers groups which are \PD{n} over arbitrary principal ideal domains and he uses a ``cell-trading argument'' in his proof of \cite[Proposition 2.7]{B96}.
Because we do not have a cocompact action, we cannot mimic this argument.
So, we appeal to the Eilenberg-Ganea Theorem which restricts our statements to integer coefficients.
\end{remark}

For the remainder of this section we assume that $(G,\mc{P})$ is a type $F$ relatively hyperbolic pair, and $\mc{X}=\mc{X}(N)$ where $N$ is large enough so that $\mc{X}(N)$ is contractible. Additionally, we will suppress integer coefficients.  We note the following corollary of Theorem \ref{thm:maintheorem} and \cite[Theorem 6.2]{BE78}.
\begin{cor}
  Suppose $(G,\mc{P})$ is type $F$ and relatively hyperbolic.  Then $(G,\mc{P})$ is \PD{n} if and only if $\partial(G,\mc{P})$ is a \v{C}ech cohomology $(n-1)$--sphere.
\end{cor}

The remainder of this section is therefore devoted to showing that if $(G,\mc{P})$ is \PD{n}, then $\partial(G,\mc{P})$ is a homology $(n-1)$--manifold.  In outline, we follow the proof of \cite[Theorem 2.8]{B96}.  The chief difference is that the cellular chain complex and compactly supported cellular cochain complex of $\mc{X}$ are not finitely generated as $\Z G$--modules.  We are nonetheless be able to show they are \emph{regular} in the sense of Definition \ref{def:regular} (cf. \cite[Definition 2.5]{B96}).  In the proof of Theorem \ref{thm:homologymanifold}, we are then able to argue as Bestvina does that the boundary is a homology manifold.

The case of \PD{2} pairs is well understood by work of Eckmann-M\"uller (covering the case that $\mc{P}\neq\emptyset$) and Eckmann-Linnell (covering the absolute case).
\begin{theorem}\label{thm:pd2} \cite[4.3]{EM80}\cite[Theorem 2]{EL83}
  If $(G,\mc{P})$ is \PD{2}, then $G=\pi_1\Sigma$ for some compact surface.  If $\mc{P}$ is empty, $\Sigma$ is closed, and otherwise the elements of $\mc{P}$ are the fundamental groups of the boundary components of $\Sigma$.
\end{theorem}
  If such a pair is relatively hyperbolic, its Bowditch boundary is $S^1$.  We can therefore make the following assumption:
\begin{assumption}
 $  n\geq 3 $.
\end{assumption}

\begin{prop}
  If $(G,\mc{P})$ is a \PD{n} pair, we can assume the complex $\mc{X}$ is $n$--dimensional.
\end{prop}
\begin{proof}
  As explained in \cite{BE78}, both $G$ and each $P\in \mc{P}$ are $(n-1)$--dimensional duality groups.  Each $P\in \mc{P}$ is moreover a \PD{n-1} group.  In particular $G$ has cohomological dimension $n$, and by the Eilenberg-Ganea Theorem admits a classifying space (which we can assume is simplicial) $B_G$ with $\dim(B_G) =\max\{3,n-1\}$.  Thus the cocompact part of $\mc{X}$ is at most $n$--dimensional.

  We claim that each $P\in \mc{P}$ has a classifying space of dimension $(n-1)$.  If $n\geq 4$, this follows from Eilenberg-Ganea.  If $n = 3$, then it follows from Theorem \ref{thm:pd2}.  It follows that the horoballs of $\mc{X}$ can be taken to be $n$--dimensional as well.
\end{proof}

We are therefore justified in making the following:
\begin{assumption}\label{assumption:dim}
  $\mc{X}$ is $n$--dimensional.
\end{assumption}

We note the following corollary (this also follows from Theorem \ref{topdimthm}).
\begin{cor}\label{cor:top dim}
  The topological dimension of $\partial(G,\mc{P})$ is $n-1$.
\end{cor}
\begin{proof}
  By Theorem \ref{thm:zset}, $\partial(G,\mc{P})$ is a $Z$--set compactification of $\mc{X}$.
  
  By Proposition 2.6 of \cite{BM91} (see \cite{GuilbaultTirel} for an alternate proof), the dimension of $\partial(G,\mc{P})$ is strictly less than the dimension of $\mc{X}$, so $\dim \partial(G,\mc{P})\leq n-1$.  On the other hand Theorem \ref{thm:maintheorem} gives $\check{H}^{n-1}(\partial(G,\mc{P}))\cong H^n(G,\mc{P};\Z G)\cong \Z \neq 0$, so $\dim\partial(G,\mc{P})\geq n-1$.
  \end{proof}

If $G$ acts nontrivially on $\tilde{\Z}$ it has an index $2$ subgroup $H$ which does act trivially.  Let $\mc{P}_H$ be the induced peripheral structure on $H$ as in Definition \ref{def:induced}.  Then $(H,\mc{P}_H)$ is relatively hyperbolic, with the same Bowditch boundary as $(G,\mc{P})$, by Lemma \ref{lem:finiteindex}.  Moreover, $(G,\mc{P}_H)$ is a \PD{n} pair with trivial action on $H^n(H,\mc{P}_H;\Z H)$ by \cite[Theorem 7.6]{BE78}.
So, to prove Theorem \ref{thm:pdpair}, it suffices to prove the theorem in the case that $\tilde{\Z}$ has a trivial $G$--action.

\begin{assumption}\label{assumption:op}
$\tilde{\Z}$ has a trivial $G$--action.
\end{assumption}

\begin{remark}
  Let $M$ be equal either to the $k$--chains of $\mc{X}$ or the $k$--cochains of compact support (either cellular or simplicial).  Then for each $m\in M$ there is a well-defined \emph{support} of $m$ in $\mc{X}$:
  \begin{itemize}
  \item If $M$ is the $k$--chains, and $m = \sum_{i=1}^l\lambda_i \sigma_i$ is an expression as a sum with $\sigma_i=\sigma_j$ only when $i=j$, then $\supp(m) = \cup\{\sigma_i\mid \lambda_i \neq 0\}$.
  \item  If $M$ is the compactly supported $k$--cochains, and $m\in M$, then $\supp(m) = \cup\{\sigma\mid m(\sigma)\neq 0\}$.
  \end{itemize}
  Supports have the following nice properties:
  \begin{enumerate}
  \item For each $g\in G$, $m\in M$, $\supp(gm) = g\cdot\supp(m)$.
  \item For any $m,n\in M$, $\supp(m+n)\subset \supp(m)\cup\supp(n)$.
  \end{enumerate}
  We will refer to such an $M$ as a \emph{$G$--module with supports in $\mc{X}$}.
\end{remark}
  
  The following definition of ``regularity'' is an adaptation of \cite[Definition 2.5]{B96}.  Showing that the compactly supported cochains form a regular complex will be the key to showing the boundary is a homology manifold.
  \begin{figure}[htbp]
    \centering
    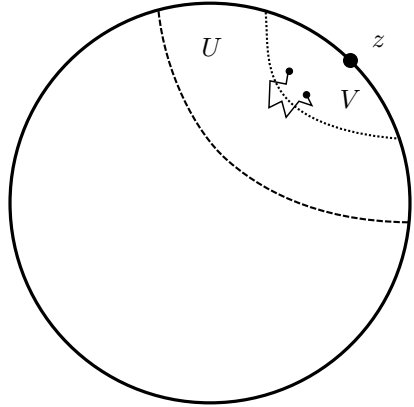
    \caption{The idea of regularity.}
    \label{fig:regular}
  \end{figure}
\begin{definition}\label{def:regular}
  Suppose that $\underline{C} = \{\cdots \to  C_{i+1}\stackrel{\partial}{\rightarrow} C_i\to \cdots\}$ is a finite length chain complex of $G$--modules with supports in $\mc{X}$.  We say $\underline{C}$ is \emph{regular} if, for every $z\in \partial(G,\mc{P})$, and every open neighborhood $U\subset \overline{\mc{X}}$ of $z$, there is a smaller open neighborhood $V$ so that whenever $c$ is an $i$--boundary with $\supp(c)\subset V$, then $c = \partial d$ for some $d$ with $\supp(d)\subset U$.
\end{definition}
The following lemma doesn't depend on $\dim\mc{X} = n$.
\begin{lemma}\label{lem: chain complex regular}
  The cellular chain complex of $\mc{X}$ is regular.
\end{lemma}
\begin{proof}
Let $z\in\partial(G,\mc{P})$.
We will denote a subset of $\bar{\mc{X}}$ as $\bar{S}$ and we will denote its intersection with $\mc{X}$ as $S$.
Letting $\bar{U}$ be an arbitrary neighborhood of $z$, we can take $\bar{W}\subset\bar{U}$ such that there is a subcomplex $L$ of $\mc{X}$ with $W\subset L\subset U$.
We can also take $\bar{V}$ such that the inclusion $V\hookrightarrow W$ induces the trivial homomorphism on $H_i$ for $i\le n$.
This follows from using Proposition \ref{prop:conditioniii} to obtain open sets $\bar{V}_n\subset\bar{V}_{n-1}\subset\cdots\subset\bar{U}$ with each $V_i\hookrightarrow V_{i-1}$ inducing the trivial map on $\pi_i$ for $i\le n$ and applying Lemma \ref{lem:HLC lemma}.
Now, if we have a cellular cycle supported in $\bar{V}$, it must be bounded by a chain in $L$.
\end{proof}

In this section and the next one we will be interested in maps between $G$--modules with supports which do not move supports too much.
  \begin{definition}\label{def:boundeddisplacement}
    Let $M$ and $N$ be $G$--modules with supports in $\mc{X}$.
    A function $f\co M\to N$  has \emph{bounded displacement} if there is a number $R>0$ so that, for all $m\in M$, $\supp(f(m))$ is contained in a cellular $R$--neighborhood of $\supp(m)$.  If we need to be specific about $R$ we say that $f$ has \emph{displacement bounded by $R$}.
  \end{definition}
  For example, the cellular boundary and coboundary maps have bounded displacement, with displacement bounded by $1$.

\subsection{Regularity of Cochains}
In this subsection, we prove that the complex of simplicial cochains with compact support of $\mc{X}$ is regular.
The proof relies on comparing the cochains with compact support to the chains and controlling the differentials.
We will adopt the following notation for the remainder of this section.

\begin{notation}
Let $j>0$ be an integer.
We will use $\mc{X}_{\le j}$ and $\mc{X}_{<j}$ to denote the subspace of depth at most $j$ and the subspace of depth less than $j$.
Similarly, $\mc{X}_{\ge j}$ and $\mc{X}_{>j}$ will denote the subspace of depth at least $j$ and the subspace of depth greater than $j$.
We will use $\mc{X}_j$ to denote the subspace of depth $j$.
\end{notation}

\begin{lemma}\label{lem: inclusion of open horoball}
The inclusion $(j,\infty)\times E_P\rightarrow\mc{X}$ induces isomorphisms on $H_c^*$ and $H_*$.
\end{lemma}
\begin{proof}
For $H_*$ the assertion is trivial.  For $H_c^*$ we make heavy use of the long exact sequences from Theorem \ref{thm: LES cohom with compact support} in the Appendix.

We will describe arrows which make the following diagram commute, such that all vertical arrows are isomorphisms.
\begin{equation}\label{diag: inclusion of open horoball}
\begin{tikzcd}
0\arrow{r}&H_c^{n-1}(\mc{X}_{\le j})\arrow{r}&H_c^n(\mc{X}_{>j})\arrow{r}&H_c^n(\mathcal{X})\arrow{r}&0\\
0\arrow{r}&H_c^{n-1}(\mc{X}_{\le j})\arrow{r}\arrow{d}\arrow{u}&H_c^{n-1}(\mc{X}_j)\arrow{r}\arrow{d}\arrow{u}{\Phi}&H_c^n(\mc{X}_{<j})\arrow{r}\arrow{d}\arrow{u}{\Psi}&0\\
0\arrow{r}&H^{n-1}(G;\Z G)\arrow{r}\arrow{d}&H^{n-1}(\mathcal{P};\Z G)\arrow{r}\arrow{d}&H^n(G,\mathcal{P};\Z G)\arrow{r}\arrow{d}&0\\
0\arrow{r}&H_1(G,\mathcal{P};\Z G)\arrow{r}&H_0(\mathcal{P};\Z G)\arrow{r}&H_0(G;\Z G)\arrow{r}&0
\end{tikzcd}
\end{equation}
The first row is from the long exact sequence given by the pair $(\mathcal{X},\mc{X}_{\le j})$ and second row is from the long exact sequence given by the pair $(\mc{X}_{\le j},\mc{X}_j)$.

We next choose isomorphisms $\Phi$ and $\Psi$ so that the top right square in the diagram commutes.
Then there will exist a unique isomorphism at the top left completing that commutative square.  (In fact this isomorphism is the identity, but this is not important for our argument.)
By considering the pair $(\mc{X},\mc{X}_j)$ we obtain the exact sequence
\begin{equation}\label{eq: aux ses}
0\rightarrow H_c^{n-1}(\mc{X}_j)\xrightarrow{\delta}H_c^n(\mathcal{X}_{<j}\sqcup\mc{X}_{>j})\xrightarrow{f} H_c^n(\mathcal{X})\rightarrow0
\end{equation}
Since $H_c^n(\mc{X}_{<j}\sqcup\mc{X}_{>j})\cong H_c^n(\mc{X}_{<j})\oplus H_c^n(\mc{X}_{>j})$ there are maps $p_1$ and $p_2$ projecting onto the summands and sections $\iota_1$ and $\iota_2$.
We obtain maps $H_c^{n-1}(\mc{X}_j)\rightarrow H_c^n(\mc{X})$ which factor through $H_c^n(\mc{X}_{>j})$ and $H_c^n(\mc{X}_<j)$ by taking $-f\circ\iota_1\circ p_1\circ\delta$ and $f\circ\iota_2\circ p_2\circ\delta$.

Note that $p_1\circ\delta$ agrees with the map $H_c^{n-1}(\mc{X}_j)\rightarrow H_c^n(\mc{X}_{<j})$ in the second row of \eqref{diag: inclusion of open horoball}.
We define $\Psi=-f\circ\iota_1$.  The long exact sequence for $(\mc{X},\mc{X}_{\ge j})$ together with Corollary \ref{cor: cohom of product with half open interval} shows that $f\circ\iota_1$ is an isomorphism, so $\Psi$ is also an isomorphism.

Similarly we define $\Phi = p_2\circ\delta$.  The long exact sequence for $(\mc{X}_{\geq j},\mc{X}_j)$ together with Corollary \ref{cor: cohom of product with half open interval} shows that $\Phi$ is an isomorphism.
The composition $f\circ\iota_2$ agrees with the map $H_c^n(\mc{X}_{>j})\rightarrow H_c^n(\mc{X})$ in the top row of the diagram.
By exactness of \eqref{eq: aux ses}, we have that $f\circ\iota_1\circ p_1\circ\delta+f\circ\iota_2\circ p_2\circ\delta=0$ so the square in the top right hand corner of \eqref{diag: inclusion of open horoball} commutes.

The vertical arrows between the second and third rows are the isomorphisms from \ref{add: part of main thm proof}. 

The vertical arrows between the last two rows are the isomorphisms from Theorem \ref{thm: pd n isomorphism}.

Since the bottom row is isomorphic to the short exact sequence
\[
0\rightarrow\Delta\rightarrow\Z G/\mathcal{P}\xrightarrow{\epsilon}\Z\rightarrow0,
\]
so is the top row.

The inclusion of a component $(j,\infty)\times E_P\hookrightarrow\mc{X}_{>j}$ induces the inclusion of a $\Z$ summand into $\Z G/\mathcal{P}$, so the composition $H_c^n((j,\infty)\times E_P)\rightarrow H_c^n(\mc{X}_{>j})\rightarrow H_c^n(\mathcal{X})$ is an isomorphism.
\end{proof}

Lemma \ref{lem: inclusion of open horoball} is about singular homology and singular cohomology with compact supports but we would like to work with simplicial chains and cochains.
Fix a homeomorphism $\rho\co \R\to (0,\infty)$ which restricts to the identity on $[1,\infty)$.  Now, for each $P\in \mc{P}$, give $\R\times E_P$ a $\Z\times P$--equivariant simplicial structure so that $(\rho,\mathbf{1}_{E_P})|_{[1,\infty)\times E_P}$ is a simplicial inclusion into $\mc{X}$.  Extending $G$--equivariantly, we get a simplicial structure on all of $\mc{X}_{>0}$.

Let $C_{i,\Delta}$ denote simplicial chains and let $C_{i,\sigma}$ denote singular chains.
Consider the following maps, where the first and third maps are the inclusions of the simplicial chains into singular chains and the middle map is induced by the inclusion $\mc{X}_{>0}\rightarrow\mc{X}$.
\begin{equation}\label{eq: simplicial singular}
C_{i,\Delta}(\mc{X}_{>0})\rightarrow C_{i,\sigma}(\mc{X}_{>0})\rightarrow C_{i,\sigma}(\mc{X})\leftarrow C_{i,\Delta}(\mc{X})
\end{equation}
A simplicial chain in $C_{i,\Delta}(\mc{X}_{>0})$ with support in $\mc{X}_{\ge1}$ will get mapped to a chain in the image of $C_{i,\Delta}(\mc{X})$, allowing us to identify it with a simplicial chain in $\mc{X}$.
There are similar maps on cochains with compact support.

There are also the following commuting diagrams of simplicial chain and cochain groups.
\[
\begin{tikzcd}
C_i(\mc{X})\arrow{r}{\partial}&C_{i-1}(\mc{X})\\
C_i(\mc{X}_{\ge1})\arrow{r}{\partial}\arrow{u}\arrow{d}&C_{i-1}(\mc{X}_{\ge1})\arrow{u}\arrow{d}\\
C_i(\mc{X}_{>0})\arrow{r}{\partial}&C_{i-1}(\mc{X}_{>0})
\end{tikzcd} \hspace{1cm}
\begin{tikzcd}
C_c^i(\mc{X})\arrow{r}{\delta}&C_c^{i+1}(\mc{X})\\
C_c^i(\mc{X}_{\ge n-i+1})\arrow{r}{\delta}\arrow{u}\arrow{d}&C_c^{i+1}(\mc{X}_{\ge n-i+1})\arrow{u}\arrow{d}\\
C_c^i(\mc{X}_{>0})\arrow{r}{\delta}&C_c^{i+1}(\mc{X}_{>0})
\end{tikzcd}
\]
Commutativity of these diagrams allow us to identify cycles in $\mc{X}$ whose supports have sufficiently large depth with cycles in $\mc{X}_{>0}$, and similarly for cocycles with compact support.
Lemma \ref{lem: inclusion of open horoball} implies that the maps in (\ref{eq: simplicial singular}) induce isomorphisms on homology, as do the corresponding maps on cochains with compact support.
In particular a cycle in $\mc{X}$ with support in $\mc{X}_{>0}$ is a boundary if and only if it is a boundary in $\mc{X}_{>0}$ (and similarly for cocycles with compact support).

See Definition \ref{def:boundeddisplacement} for the definition of bounded displacement.
\begin{lemma}\label{lem: lifting maps of bounded displacement}
  Let $M$ be either $C_i(\mc{X})$ or $C_c^{n-i}(\mc{X})$ (the simplicial chains and cochains with compact support)
Let $(C_*,d)$ denote either the chain complex $(C_*(\mc{X}),\partial)$ or $(C_c^{n-*}(\mc{X}),\delta)$.
Suppose we have the following diagram of $\Z G$--modules.
\[
\begin{tikzcd}
M\arrow{rd}[above right]{f'}&\\
C_j\arrow{r}[above]{d}&C_{j-1}
\end{tikzcd}
\]
Suppose that the image of $f'$ is contained in the image of $d$.
If $f'$ has bounded displacement, then there is a map $f:M\rightarrow C_j$ of bounded displacement making the diagram commute.
\end{lemma}
\begin{proof}
(When using $(C_*,d)$, we will call elements in the image of $d$ ``boundaries'' even though they may be coboundaries.  Similarly, we will call arbitrary elements of $C_j$ ``chains'' even though they may be cochains.)

  Let $\mc{B}$ be a basis for $M$ (i.e. $\mc{B}$ contains one simplex or dual simplex per $G$--orbit).  
Since $f'$ has bounded displacement, $f'(e)$ lies in a horoball $[1,\infty)\times E_P$ for all but finitely many $e\in \mc{B}$.  Let $\mc{B}'\subset \mc{B}$ consist of those $e$ for which $f'(e)$ lies in a horoball.

For $e\in \mc{B}'$, we can identify $f'(e)$ with a chain in $\R\times E_P$, using the identification described in the text before the lemma. 
Then, Lemma \ref{lem: inclusion of open horoball} implies that $f'(e)$ represents a boundary in $\mc{X}$ if and only if it represents a boundary in $(0,\infty)\times E_P\cong \R\times E_P$.
Since $f'$ has bounded displacement, there are only finitely many boundaries in $\R\times E_P$ up to the $\Z\times P$--action that will be identified with some $f'(e)$ where $e\in \mc{B}'$.
In particular, there is a $K$ such that, when considered as a boundary of $(0,\infty)\times E_P$, $f'(e)$ is bounded by an element $g(e)$ supported in a cellular $K$-neighborhood of $\op{supp}f'(e)$.

Now we see that all but finitely many $e\in \mc{B}'$ have $f'(e)$ supported in $[K+1,\infty)\times E_P$.  Let $\mc{B}''\subset \mc{B}'$ be the set of such $e$.
For $e\in \mc{B}''$, we can identify $g(e)$ with an element of $C_j$ and define $f(e)=g(e)$.

For the finitely many $e$ in $\mc{B}\setminus \mc{B}''$, we can define $f(e)$ to be any element making the triangle commute.
\end{proof}

\begin{lemma}\label{lem: cocycles distributed throughout horoballs}
There is a constant $K$ such that, for every vertex $v$ of $\mc{X}$, there exists a cocycle $\phii\in C_c^n(\mc{X})$ supported in a cellular $K$-neighborhood of $v$ that represents a generator of $H_c^i(\mc{X})$.
\end{lemma}
\begin{proof}
Suppose $v$ has depth $j$.
Consider the map $(j,\infty)\times E_P\rightarrow\mc{X}$ in Lemma \ref{lem: inclusion of open horoball} above.
Let $B$ denote the complement of the $(j,\infty)\times E_P$ and let $\mc{H}$ denote $[j,\infty)\times E_P$ considered as a subcomplex of $\mc{X}$.
There is the short exact sequence of (simplicial) cochain complexes.
\[
0\rightarrow C_c^i(\mc{X})\rightarrow C_c^i(B)\oplus C_c^i(\mc{H})\rightarrow C_c^i(E_P)\rightarrow0
\]
Here, we are considering $E_P$ as the subcomplex $\{j\}\times E_P=B\cap\mc{H}$.

Because $(j,\infty)\times E_P\rightarrow\mc{X}$ induces isomorphisms on cohomology with compact support, $H_c^i(B)$ vanishes.
Since $\mc{H}$ is a product with $[j,\infty)$, Corollary \ref{cor: cohom of product with half open interval} implies that $H_c^i(\mc{H})$ also vanishes.
Therefore, the coboundary map $\delta:H^i(E_P)\rightarrow H^{i+1}(\mc{X})$ is an isomorphism.
Now, fix a cocycle $\psi$ that represents a generator of $H^{n-1}(E_P)$.
By translating via the action of $P$, we can assume $\psi$ is supported in a cellular $K'$ neighborhood of $v$ for some constant $K'$ independent of $v$.

We claim that a representative $\phii$ of $\delta[\psi]$ can be chosen so that the assignment $\psi\mapsto\phii$ has bounded displacement.
Indeed, let $(\psi,0)\in C_c^{n-1}(B)\oplus C_c^{n-1}(\mc{H})$.
Then, letting $\phii$ be the element mapping to $(\delta\psi,0)\in C_c^n(B)\oplus C_c^n(\mc{H})$ gives the desired assignment.
Because this assignment has bounded displacement, the result follows.
\end{proof}

\begin{prop}\label{prop: regularity}
    The cellular cochain complex of $\mc{X}$ is regular.
\end{prop}
\begin{proof}
  We adapt the proof of \cite[Proposition 2.7]{B96} to the space $\mc{X}$.
  
  By Assumption \ref{assumption:dim},  $\dim(\mathcal{X}) = n$.  Moreover the compactly supported cohomology of $\mc{X}$ is zero except in dimension $n$, where it is $\Z$.  By Assumption \ref{assumption:op}, the $G$--action on $\Z$ is trivial.

  By our assumptions we have $H_c^k(\mc{X})$ is $\Z$ when $k=n$ and $0$ otherwise.  Moreover $\mc{X}$ is contractible.  We therefore get free $\Z G$ resolutions of $\Z$,
\[
\begin{tikzcd}
0\arrow{r}& C_c^0(\mathcal{X})\arrow{r}&\cdots\arrow{r}& C_c^n(\mathcal{X})\arrow{r}&\Z\arrow{r}&0
\end{tikzcd}
\]
and
\[
\begin{tikzcd}
0\arrow{r}& C_n(\mathcal{X})\arrow{r}&\cdots\arrow{r}& C_0(\mathcal{X})\arrow{r}&\Z\arrow{r}&0.
\end{tikzcd}
\]

We construct chain maps $f:C_c^{n-i}(\mathcal{X})\rightarrow C_i(\mathcal{X})$, $g:C_i(\mathcal{X})\rightarrow C_c^{n-i}(\mathcal{X})$ and a homotopy $h$ between $g\circ f$ and the identity on $C_c^*(\mathcal{X})$ such that each of these maps has bounded displacement.

We first define $f$ on a the natural basis of cochains dual to individual cells.  Fix a generator $\alpha$ for $H^n_c(\mc{X})$.  Let $e^*$ be the cochain dual to the $n$--simplex $e$ (i.e. $e^*(e)=1$, but $e^*(e')=0$ for any cell $e'\neq e$).  There are no $(n+1)$--cochains, so $e^*$ is a cocycle, representing $k(e) \alpha$ where $k(e)\in\Z$.
The $n$--simplex $e$ is the image of an embedding from the standard simplex $\Delta^n$ into $\mc{X}$.  This standard simplex is the convex hull of the standard unit vectors $v_0,\ldots,v_n\in \R^{n+1}$.  Let $p(e)$ be the vertex which is the image of $v_0$.  Define $f_0(e^*) = k(e)p(e)$.
Since the simplicial structure on $\mc{X}$ comes from a $\Delta$--complex structure on the quotient $C = \leftQ{\mc{X}}{G}$, this defines a map of $\Z G$--modules
\[ f_0\co C_c^n(\mc{X})\to C_0(\mc{X}). \]
The map $f_0$ clearly has bounded displacement.
Define $f$ by applying Lemma \ref{lem: lifting maps of bounded displacement} inductively to the following diagram.
\[
\begin{tikzcd}
C_c^{n-i}(\mc{X})\arrow{rd}{f_{i-1}\circ\delta}&\\C_{i}(\mc{X})\arrow{r}{\partial}&C_{i-1}(\mc{X})
\end{tikzcd}
\]

By Lemma \ref{lem: cocycles distributed throughout horoballs} there is a constant $K$ such that, for each vertex there is a cocycle representing $1\in H_c^n(\mathcal{X})$ supported in a $K$-neighborhood of the vertex.
This allows us to define $g_0:C_0(\mathcal{X})\rightarrow C_c^n(\mathcal{X})$.
Now $g$ can be extended to a map of bounded displacement on $C_c^*(\mathcal{X})$ by applying Lemma \ref{lem: lifting maps of bounded displacement} inductively to the following diagram.
\[
\begin{tikzcd}
C_i(\mc{X})\arrow{rd}{g_{i-1}\circ\delta}&\\C_c^{n-i}(\mc{X})\arrow{r}{\delta}&C_c^{n-i+1}(\mc{X})
\end{tikzcd}
\]

For the homotopy, we need $h\delta+\delta h=\op{Id}-g\circ f$.
This can be done by setting the map $\Z\rightarrow C_c^n(\mathcal{X})$ to be $0$ and applying Lemma \ref{lem: lifting maps of bounded displacement} to the following diagram.
\[
\begin{tikzcd}
C_c^i(\mc{X})\arrow{rd}{Id-g\circ f-h\delta}&\\C_c^{i-1}(\mc{X})\arrow{r}{\delta}&C_c^i(\mc{X})
\end{tikzcd}
\]
There is some $M>0$ so that all the maps $f_i$, $g_i$, $\delta$, and $h$ have displacement bounded by $M$.

  Let $z\in\partial(G,\mc{P})$ and let $U$ be an open neighborhood of $z$ in $\overline{\mc{X}}$.  Let $W$ be the subcomplex of $\mc{X}$ consisting of those simplices whose $3M$--cellular neighborhoods are completely contained in $U$, and let $U_1$ be the interior of $W$.

  Let $V_1\subset U_1$ be an open neighborhood of $z$ so that every chain with support in $V_1$ which is an $\mc{X}$--boundary  is the boundary of a chain with support in $U_1$.  Finally let $V\subset V_1$ be a neighborhood of $z$ so that the $3M$--cellular neighborhood of every simplex of $V$ is contained in $V_1$

  Now suppose that $b = \delta\phii$ has support in $V$, where $\phii$ is a simplicial $(k-1)$--cochain in $\mc{X}$.  Then $f(b)$ is a boundary with support in $V_1$, so $f(b) = \partial \sigma$ for a chain $\sigma$ with support in $U_1$.  We have $\delta g (\sigma)=g\circ f(b) = b-\delta h b$, so $b = \delta(g(\sigma)+h(b))$.  Since $\sigma$ and $b$ both have support inside $U_1$, the cochain $g(\sigma)+h(b)$ has support inside $U$.
\end{proof}

The following lemma is a rephrasing of Proposition \ref{prop: regularity}.

\begin{lemma}\label{lemma: regularity 2}
Let $\{\tilde{U}_i\}_{i\in\N}$ be a neighborhood basis of $z\in\overline{\mc{X}}$ such that $\mc{X}\setminus U_i$ is a subcomplex for each $i$.
Let $U_i=\tilde{U}_i\setminus\partial(G,\mc{P})$.
Then, for each $i$, there is a $j>i$ such that, if an element $[\phii]\in H_c^k(U_j)$ is sent to $0$ under $H_c^k(U_j)\rightarrow H_c^k(\mc{X})$, then it is sent to $0$ under $H_c^k(U_j)\rightarrow H_c^k(U_i)$
\end{lemma}
\begin{proof}
Identify $H_c^*(U_j)$ with the cohomology of the kernel of $C_c^*(\mc{X})\rightarrow C_c^*(\mc{X}\setminus U_i)$ and apply Proposition \ref{prop: regularity}.
\end{proof}

  The following completes the proof of Theorem \ref{thm:pdpair}.
\begin{theorem}\label{thm:homologymanifold}
  If $(G,\mc{P})$ is a type $F$ relatively hyperbolic \PD{n} pair, then $\partial(G,\mc{P})$ is a homology $(n-1)$--manifold.
\end{theorem}
\begin{proof}
  Once we have the regularity of the compactly supported cochains, the proof follows exactly as Bestvina's proof that the boundary of a hyperbolic \PD{n} group is a homology $(n-1)$--sphere (see \cite[2.8]{B96}).  For completeness we give the argument, filling in a few details.
  
  To align notation with Bestvina's, write $Z = \partial(G,\mc{P})$, $X=\mc{X}$, and $\tilde{X} = \overline{X} = X\cup Z$. In this proof, $H_k(-)$ will denote Steenrod homology and $H_k^{LF}(-)$ will denote locally finite homology.  For an exposition of these homology theories see \cite{Ferry}.
  
  We aim to show that, for any point $z\in Z$,
  \begin{equation*}
   H_k (Z,Z\setminus\{z\})\cong
    \begin{cases}
      \Z & k=n-1\\
      0 & \mbox{otherwise}
    \end{cases}  
  \end{equation*}
  We have shown (Theorem \ref{thm:zset}) that $Z$ is a $Z$-set in $\tilde{X}$, which is an absolute retract.
  In this setting, we have the following two facts, special cases of \cite[Proposition 1.8 and Remark 1.9]{B96}:
  \begin{itemize}
    \item If $\{\tilde{U_i}\}_{i\in \N}$ is a neighborhood basis in $\tilde{X}$  of $z\in Z$, then \[H_k(Z,Z\setminus\{z\})= \varinjlim H_{k+1}^{\mathrm{LF}}(U_i).\]
    \item (Universal coefficients.)  For each $U_i$ there is a short exact sequence
      \begin{equation}\label{uct} 0 \to \op{Ext}(H_c^{k+2}(U_i);\Z)\to H_{k+1}^{\mathrm{LF}}(U_i) \to \op{Hom}(H_c^{k+1}(U_i);\Z)\to 0.
      \end{equation}
    \end{itemize}

    Since a direct limit of exact sequences is exact, we can use a direct limit of the short exact sequences \eqref{uct} to compute $H_k(Z,Z\setminus\{z\})$.  
    But the limits $\varinjlim\op{Hom}(-,\Z)$ and $\varinjlim\op{Ext}(-,\Z)$ depend only on the inverse system up to pro-isomorphism (see \cite[Chapter 11]{Geoghegan}).
    Indeed, if $F$ is a contravariant functor, it sends pro-isomorphic systems to ind-isomorphic systems and colimits of ind-isomorphic systems are isomorphic.
  It is therefore enough to prove that the inverse system $\{H_c^k(U_i)\}_{i\in\N}$ is pro-trivial when $k\neq n$, and pro-isomorphic to $\{\Z\}$ when $k = n$.  In other words, we want to show the inverse system $\{H_c^k(U_i)\}_{i\in\N}$ is pro-isomorphic to the inverse system consisting of a single group, $\{H_c^k(\mc{X})\}$.
  
To show that the two systems are pro-isomorphic, we give maps $p_i:H_c^k(\mc{X})\rightarrow H_c^k(U_i)$ and $q:H_c^k(U_0)\rightarrow H_c^k(\mc{X})$.
When $H_c^k(\mc{X})$ is trivial, then these are zero.
In general we take $q$ to be the map induced by the inclusion $U_0\subseteq\mc{X}$.
Let $\alpha:\N\rightarrow\N$ be the assignment $i\mapsto j$ of Lemma \ref{lemma: regularity 2}.
For $k=n$, we have $H_c^n(\mc{X})\cong\Z$.
A cochain representing the generator can be translated into each $U_i$ by the action of $G$.
Thus, the restriction $H_c^n(\mc{X})\rightarrow H_c^n(\mc{X}\setminus U_i)$ is zero and, using the long exact sequence from Theorem \ref{thm: LES cohom with compact support}, $H_c^n(U_i)\rightarrow H_c^n(\mc{X})$ is surjective.
Since $H_c^n(\mc{X})$ is free, this surjection admits a section, which we will denote $s_i$.
Let $p_i$ be the composite $H_c^n(\mc{X})\xrightarrow{s_{\alpha(i)}}H_c^n(U_{\alpha(i)})\rightarrow H_c^n(U_i)$.

We must check that the $p_i$ commute with the maps in the system $\{H_c^n(U_i)\}$.
Consider the following triangle.
\[
\begin{tikzcd}
H_c^n(\mc{X})\arrow{d}{p_j}\arrow{rd}{p_i}&\\
H_c^n(U_j)\arrow{r}{\iota_*}&H_c^n(U_i)
\end{tikzcd}
\]
Let $\phii$ represent a generator of $H_c^n(\mc{X})$.
Then, $\iota_*\circ p_j[\phii]-p_i[\phii]$ is represented by a cochain supported in $U_{\alpha(i)}$.
Moreover, this cochain is a coboundary in $\mc{X}$ so, by definition of $\alpha$, $\iota_*\circ p_j[\phii]-p_i[\phii]=0$.

That these maps give a pro-isomorphism can be checked using Lemma \ref{lemma: regularity 2}.

  To conclude, $Z$ is $(n-1)$--dimensional (Corollary \ref{cor:top dim}) and has the local homology of an $(n-1)$--manifold at every point, so it is a homology $(n-1)$--manifold.
\end{proof}

\section{Topological dimension of the boundary.}\label{sec:dimension}
In \cite{BM91} it is established that for a hyperbolic group $G$, the topological dimension of $\partial G$ is exactly one less than $\max\{n\mid H^n(G;\Z G)\neq 0\}$.  We extend this to the relative setting in a special case:
\begin{theorem}\label{topdimthm}
  Let $(G,\mc{P})$ be type $F$ and relatively hyperbolic.  Suppose further that $\op{cd}(G)<\op{cd}(G,\mathcal{P})$.
Then
\[
\dim(\partial(G,\mathcal{P}))=\op{cd}(G,\mathcal{P})-1
\]
\end{theorem}

\begin{remark}
Under the hypotheses of Theorem \ref{topdimthm}, it must be the case that \[\op{cd}(G)=\max_{P\in\mc{P}}\op{cd}(P)=\op{cd}(G,\mc{P})-1.\]
Moreover, if $\op{cd}(G)>2$, we may apply the Eilenberg-Ganea Theorem as in the previous section to obtain Theorem \ref{topdimthm}.
The proof we give in this section follows \cite{BM91} and applies also to the case $\op{cd}(G)=2$.
\end{remark}

As in the previous section, we suppress integer coefficients. In the proof of \cite[Corollary 1.4]{BM91}, Bestvina and Mess prove the following.

\begin{lemma}\label{lem: dim of boundary/reg cochains}
Suppose $Z$ is a $Z$-set in $\bar{X}$ such that $X=\bar{X}\setminus Z$ is a locally finite CW complex of dimension $N$.
Let $n-1=\max\{k\in\Z_{\geq 0}\mid\check{H}^k(Z)\neq0\}$ and let $C_c^k(X)$ denote the cellular $k$-cochains with compact support.
If the cochain complex
\[
C_c^n(X)\rightarrow\cdots\rightarrow C_c^N(X)\rightarrow0
\]
is regular, then $\dim(Z)=n-1$.
\end{lemma}

\begin{proof}[Proof of Theorem \ref{topdimthm}]
Let $n=\op{cd}(G,\mathcal{P})$ and let $N=\dim(\mathcal{X})$.
Note that $\op{cd}(P)<n$ for all $P\in\mathcal{P}$.
Let $C_c^k(\mc{X})$ denote the cellular cochains of $\mc{X}$ with compact support.  Each horoball of $\mc{X}$ is a product $[0,\infty)\times E_P$ for some simplicial complex $E_P$.  We emphasize that we are using the product \emph{cellulation} of the horoballs, and not, as before, the product simplicial structure.  Namely, each cell in a horoball is a product of a cell of $E_P$ with a cell of $[0,\infty)$.

We verify the hypotheses of Lemma \ref{lem: dim of boundary/reg cochains}, showing that the truncated complex of compactly supported cochains 
\[ 
C_c^n(\mc{X})\rightarrow\cdots\rightarrow C_c^N(\mc{X})\rightarrow0
\] is regular.

We must show, for each $k>n$, there is an $M>0$ such that, if $\phii\in C_c^k(\mathcal{X})$ is a coboundary, then it is the coboundary of some $\psi\in C_c^{k-1}(\mathcal{X})$ supported in a cellular $M$--neighborhood of $\phii$.

Let $X=\mc{X}_{\le1}$ and let $Y=\mc{X}_1$.
The compactly supported $k$--cochains of $\mathcal{X}$ admit the following decomposition
\[
C_c^k(\mathcal{X})\cong C_c^k(X)\oplus\left(\bigoplus_{m\in\N}C_c^k(Y)\right)\oplus\left(\bigoplus_{m\in\N}C_c^{k-1}(Y)\right).
\]
Letting $e$ be a cell in either $X$ or $Y$ and $e^*$ the cochain that sends $e$ to $1$, $\delta:C_c^k(\mathcal{X})\rightarrow C_c^{k-1}(\mathcal{X})$ is defined by
\begin{align*}
\delta(e^*,0,0)&=\begin{cases}
(\delta e^*,0,-e^*_1)&e\in Y\\
(\delta e^*,0,0)&e\notin Y
\end{cases}\\
\delta(0,e^*_m,0)&=(0,(\delta e^*)_m,e^*_m-e^*_{m+1})\\\delta(0,0,e^*_m)&=(0,0,-\delta e^*_m)
\end{align*}
where $e^*_m$ denotes the cocycle $e^*$ in the $m$-th summand of $\bigoplus_{m\in\N} C_c^k(Y)$.
By our assumption that $\op{cd}(G)<n$ and $\op{cd}(P)<n$ for all $P\in \mc{P}$, there are acyclic free $\Z G$--complexes (free resolutions of $0$)
\begin{align*}
...\rightarrow F_{N-n+3}\rightarrow F_{N-n+2}\rightarrow C_c^{n-1}(X)\rightarrow C_c^{n}(X)\rightarrow...\rightarrow C_c^N(X)\rightarrow0\\
...\rightarrow F_{N-n+3}'\rightarrow F_{N-n+2}'\rightarrow C_c^{n-1}(Y)\rightarrow C_c^n(Y)\rightarrow...\rightarrow C_c^N(Y)\rightarrow0
\end{align*}
For these resolutions, there are chain homotopies $h_X$ and $h_Y$ between the identity and $0$ (i.e. $h\delta+\delta h=\op{Id}$).
Since each $C_c^k(X)$ and $C_c^k(Y)$ appearing in these resolutions is finitely generated, there is an $M$ so that $h_X$ and $h_Y$ have displacement bounded by $M$.

For $k>n$, define $:C_c^k(\mathcal{X})\rightarrow C_c^{k-1}(\mathcal{X})$ by $H(\phii,\phii',\phii'')=(h_X\phii,h_Y\phii',-h_Y\phii'')$.
A calculation shows that $H\delta+\delta H=\op{Id}$.
Now, if $\phii$ is a cocycle in $C_c^k(\mathcal{X})$ and $k>n$, then $\delta H\phii=\phii$.
The claim follows from the fact that $d_H(\op{supp}(H\psi),\op{supp}(\psi))<N$.

With this claim, the result follows from Lemma \ref{lem: dim of boundary/reg cochains}.
\end{proof}

Under the assumptions of Theorem \ref{topdimthm}, the group $G$ is type $F$ and thus $\op{cd}(G,\mc{P})$ is equal to  $\max\{n\mid H^n(G,\mc{P};\Z G)\neq 0\}$ (see Proposition \ref{prop:kap}).  
\begin{conjecture}\label{conj:dimension}
  Let $(G,\mc{P})$ be relatively hyperbolic and type $F$.  Then
  \[ \dim(\partial(G,\mc{P})) = \max\{n\mid H^{n}(G,\mc{P};\Z G)\neq 0\}-1.\]
\end{conjecture}
We are not sure if the equality should hold without the type $F$ assumption.

\appendix
\section{Cohomology with compact support}

The purpose of this appendix is to establish the long exact sequence for cohomology with compact support.
This long exact sequence is well documented in texts on sheaf cohomology but we will need singular and simplicial statements to apply this to group cohomology.
Moreover, the definition of singular cohomology with compact support in \cite{Spanier} and \cite{Hatcher} differ slightly; the definition in \cite{Spanier} is more compatible with sheaf cohomology while the definition in \cite{Hatcher} is more compatible with simplicial and cellular homology.
One of our goals is to show that the definition in \cite{Hatcher} is well behaved with respect to sheaf cohomology.
The main result of this section is Theorem \ref{thm: LES cohom with compact support}.

Fix an abelian group. In this section, we take coefficients in this abelian group or in the constant sheaf determined by this group.

\begin{definition}\label{def: simplicial cohom with compact support}
Suppose $X$ is a locally finite simplicial complex (or more generally, a locally finite $\Delta$-complex).
Then we define the \emph{simplicial $i$-cochains with compact support} to be the cochains $\phii$ such that $\phii(\sigma)=0$ for all but finitely many $i$--simplices $\sigma$.
We denote these cochains by $\tensor*[_\Delta]{C}{_c^i}(X)$.
This gives a cochain complex and we define the \emph{simplicial cohomology with compact support} $\tensor*[_\Delta]{H}{_c^i}(X)$ to be the cohomology of this complex.
\end{definition}

\begin{definition}\label{def: singular cohom with compact support}
If $X$ is a locally compact space, we define the \emph{singular $i$-cochains with compact support} to be the cochains $\phii$ such that there is a compact subset $K\subseteq X$ such that $\phii(\sigma)=0$ for all singular simplices $\sigma$ with image in $X\setminus K$.
We will denote these as $\tensor*[_\sigma]{C}{_c^i}(X)$.
This gives a cochain complex and we define the \emph{singular cohomology with compact support} $\tensor*[_\sigma]{H}{_c^i}(X)$ to be the cohomology of this complex.
\end{definition}

\begin{remark}
In the case that $X$ is a locally finite simplicial complex, singular cohomology with compact support and simplicial cohomology with compact support agree.
Moreover, the isomorphism is induced by a map of cochain complexes.
Similarly, if $X$ is a locally finite $CW$-complex, singular cohomology with compact support and cellular cohomology with compact support agree.
However, this isomorphism is not induced by a chain map.
\end{remark}

In order to relate singular cohomology with compact support to sheaf cohomology with compact support, we need to make the following definitions (see \cite[I.7]{Bredon}).

\begin{definition}\label{def: S_c and S_0}
Let $S_c^i(X)$ denote the singular $i$-cochains $\phii$ such that there exists a compact subset $K\subseteq X$ where, for each $x\in X\setminus K$, there is a open neighborhood $W$ of $x$ such that $\phii\mapsto0$ under $C^i(X)\rightarrow C^i(W)$.
Let $S_0^i(X)$ denote the singular $i$-cochains $\phii$ such that, for each $x\in X$, there is an open neighborhood $W$ of $x$ such that $\phii\mapsto0$ under $C^i(X)\rightarrow C^i(W)$.
\end{definition}

The cochains $S_c^i(X)$ (resp. $S_0^i(X)$) are the global sections with compact (resp. trivial) support of the singular cochain presheaf.
It follows from definitions that there is an inclusion $\tensor*[_\sigma]{C}{_c^i}(X)\rightarrow S_c^i(X)$.

\begin{definition}
Let $\sigma=\sum_{j=1}^ma_j\sigma_j\in \tensor*[_\sigma]{C}{_i}(X)$ be a singular chain where $\sigma_j:\Delta^i\rightarrow X$ and $\sigma_j\neq\sigma_{j'}$ when $j\neq j'$.
Then the \emph{support of $\sigma$} is the union of the images of the $\sigma_j$.
We will denote this by $\op{supp}\sigma$.
\end{definition}

We first record a result in Spanier.

\begin{lemma}\label{prop: U small chains}
Let $\mc{U}$ be an open cover of $X$.
Let $\tensor*[_\sigma]{C}{_i}(\mc{U})$ denote the $\mc{U}$-small singular $i$-chains (i.e. combinations of simplices each supported in some $U\in\mc{U}$).
Then the inclusion $\iota:\tensor*[_\sigma]{C}{_i}(\mc{U})\rightarrow\tensor*[_\sigma]{C}{_i}(X)$ is a chain homotopy equivalence.
Moreover, the inverse $g:\tensor*[_\sigma]{C}{_i}(X)\rightarrow\tensor*[_\sigma]{C}{_i}(\mc{U})$ and homotopy $h:\tensor*[_\sigma]{C}{_i}(X)\rightarrow\tensor*[_\sigma]{C}{_{i+1}}(X)$ can be taken such that, for $\sigma\in \tensor*[_\sigma]{C}{_i}(X)$, $\op{supp}g(\sigma)\subseteq\op{supp}\sigma$ and $\op{supp}h(\sigma)\subseteq\op{supp}\sigma$.
\end{lemma}

The first part of this statement is \cite[Theorem IV.4.14]{Spanier}.
The second part comes from the construction given in the proof.

\begin{prop}\label{prop: S_0 acyclic}
The complex $S_0^*(X)$ is acyclic.
\end{prop}

\begin{prop}\label{prop: singular, sheaf cohom}
Suppose $X$ is locally compact.
Then, the inclusion $\tensor*[_\sigma]{C}{_c^i}(X)\rightarrow S_c^i(X)$ induces an isomorphism on cohomology groups.
\end{prop}
\begin{proof}
Let $D^i$ denote the cokernel of $\tensor*[_\sigma]{C}{_c^i}(X)\rightarrow S_c^i(X)$.
If the cochain complex $D^i$ is acyclic, then the proposition follows.
Note that the submodule $S_0^i(X)$ surjects to $D^i$ under the map $S_c^i(X)\rightarrow D^i$.
Let $E^i$ be the kernel of the map $S_0^i(X)\rightarrow D^i$.
Since $S_0^i(X)$ is acyclic, it suffices to show that $E^i$ is acyclic.
$E^i$ consists of the cochains $\phii\in S_0^i(X)$ such that there is a compact subset $K\subseteq X$ with $\phii(\sigma)=0$ for all $\sigma$ whose image is in $X\setminus K$.

Suppose $\phii\in E^i$ is a cocycle.
Then, we may consider $\phii$ as a cocycle in $S_0^i(X)$.
There is a compact set $K$ such that $\phii(\sigma)=0$ whenever the $\op{supp}(\sigma)\subseteq X\setminus K$.
There is also an open cover $\mc{U}$ of $X$ such that $\phii(\sigma)=0$ for all $\sigma\notin C_{i,\sigma}(\mc{U})$.
Let $\iota,g$ and $h$ be as in Proposition \ref{prop: U small chains} and let $\iota^*,g^*$ and $h^*$ denote their duals on cochains.
Then, $h\partial+\partial h=\op{Id}-\iota\circ g$ so $\delta h^*+h^*\delta=\op{Id}-g^*\circ\iota^*$.
Since $\phii$ is a cocycle and $\iota^*(\phii)=0$, we obtain $\delta h^*(\phii)=\phii$.
We claim that $h^*(\phii)\in E^{i-1}$.
Indeed, if $\sigma$ is a singular $(i-1)$-simplex with $\op{supp}(\sigma)\subseteq X\setminus K$, then $\op{supp}(h(\sigma))\subseteq X\setminus K$.
Therefore, $\phii(h(\sigma))=(h^*(\phii))(\sigma)=0$.
This implies that $E^i$ is acyclic.
\end{proof}

We can now prove the following theorem.

\begin{restatable}{theorem}{LEScohomwithsupp}\label{thm: LES cohom with compact support}
Suppose $X$ is locally compact, locally contractible and Hausdorff.
Let $F\subseteq X$ be a closed, locally contractible subset such that $U:=X\setminus F$ is also locally contractible.
Then, there is the following long exact sequence.
\[
\cdots\rightarrow\tensor*[_\sigma]{H}{_c^i}(U)\rightarrow \tensor*[_\sigma]{H}{_c^i}(X)\rightarrow\tensor*[_\sigma]{H}{_c^i}(F)\rightarrow\tensor*[_\sigma]{H}{_c^{i+1}}(U)\rightarrow\cdots
\]
Moreover, if $X$ is a simplicial complex and $F$ is a subcomplex, then there is the following long exact sequence.
\[
\cdots\rightarrow \tensor*[_\sigma]{H}{_c^i}(U)\rightarrow \tensor*[_\Delta]{H}{_c^i}(X)\rightarrow\tensor*[_\Delta]{H}{_c^i}(F)\rightarrow\tensor*[_\sigma]{H}{_c^{i+1}}(U)\rightarrow\cdots
\]
\end{restatable}

\begin{proof}
For a closed subspace $F$, there is a surjection $\tensor*[_\sigma]{C}{_c^i}(X)\rightarrow\tensor*[_\sigma]{C}{_c^i}(F)$.
Let $\tensor*[_\sigma]{C}{_c^i}(X,F)$ denote the kernel of this map.
Similarly, there is a surjection $S_c^i(X)\rightarrow S_c^i(F)$ and we let $S_c^i(X,F)$ denote the kernel.
This gives the following commuting diagram with exact rows.
\begin{equation}\label{diag: singular to sheaf}
\begin{tikzcd}
0\arrow{r}&\tensor*[_\sigma]{C}{_c^i}(X,F)\arrow{r}\arrow{d}&\tensor*[_\sigma]{C}{_c^i}(X)\arrow{r}\arrow{d}&\tensor*[_\sigma]{C}{_c^i}(F)\arrow{r}\arrow{d}&0\\
0\arrow{r}&S_c^i(X,F)\arrow{r}&S_c^i(X)\arrow{r}&S_c^i(F)\arrow{r}&0
\end{tikzcd}
\end{equation}
The middle and right vertical maps in  \eqref{diag: singular to sheaf} are the inclusions of cochains and the left vertical map exists by exactness.
By Proposition \ref{prop: singular, sheaf cohom} the middle and right vertical maps induce isomorphisms on cohomology.
Thus, $\tensor*[_\sigma]{C}{_c^i}(X,F)\rightarrow S_c^i(X,F)$ induces an isomorphism on cohomology.

Since $X$ and $F$ are locally contractible and locally compact, the bottom row of Diagram \ref{diag: singular to sheaf} computes sheaf cohomology (see \cite[III.1]{Bredon}).
So, the cohomology of the complex $\tensor*[_\sigma]{C}{_c^i}(X,F)$ is isomorphic to the relative sheaf cohomology with compact support of the pair $(X,F)$ with coefficients in the constant sheaf.
By \cite[Proposition II.12.3]{Bredon}, this is isomorphic to $H_c^i(U)$.
Since $U$ is locally contractible, this is isomorphic to singular cohomology with compact support.
\end{proof}

We will need the following consequence of the proof of Theorem \ref{thm: LES cohom with compact support}.

\begin{restatable}{theorem}{RelCochainsWithSupport}\label{thm: rel cochains with support}
Let $X,F$ and $U$ be as in Theorem \ref{thm: LES cohom with compact support}.
Then the $i$-th cohomology of $\tensor*[_\sigma]{C}{_c^*}(X,F)$ is $\tensor*[_\sigma]{H}{_c^i}(U)$.
If $X$ is a simplicial complex and $F$ is a subcomplex then the $i$-th cohomology of $\tensor*[_\Delta]{C}{_c^*}(X,F)$ is $\tensor*[_\sigma]{H}{_c^i}(U)$
\end{restatable}

We will also need the following corollary of Theorem \ref{thm: LES cohom with compact support}.

\begin{cor}\label{cor: cohom of product with half open interval}
Suppose $X$ is locally compact, locally contractible and Hausdorff.
Then \[\tensor*[_\sigma]{H}{_c^*}\left([0,1)\times X\right)=0.\]
\end{cor}
\begin{proof}
The inclusion $\{1\}\times X\rightarrow [0,1]\times X$ is a proper homotopy equivalence so it induces isomorphisms on cohomology with compact support.
The result follows from the long exact sequence of Theorem \ref{thm: LES cohom with compact support}.
\end{proof}

\end{document}